\documentclass[10pt]{article}
\usepackage{amsthm}
\usepackage{amsmath}
\usepackage{amssymb}

\usepackage{todonotes}

\usepackage[shortlabels]{enumitem}

\usepackage{tikz}
\usetikzlibrary{arrows, shapes}
\usepackage[font=small, labelfont=bf]{caption}

\newtheorem{theorem}{Theorem}
\newtheorem{lemma}[theorem]{Lemma}
\newtheorem{corollary}[theorem]{Corollary}
\newtheorem{proposition}[theorem]{Proposition}

\theoremstyle{definition}
\newtheorem{definition}[theorem]{Definition}

\theoremstyle{remark}

\newtheorem{open}{Open Problem}

\renewcommand{\P}{\mathrm{P}}
\newcommand{\NP}{\ensuremath{\mathrm{NP}}}
\newcommand{\PV}{\mathrm{PV}}
\newcommand{\LPV}{L_{\PV}}
\newcommand{\N}{\mathbb{N}}
\newcommand{\true}{\forall \PV(\N)}

\newcommand{\PLS}{\mathrm{PLS}}
\newcommand{\CPLS}{\mathrm{CPLS}}
\newcommand{\APC}{\mathrm{APC}}

\newcommand{\sWPHP}{\mathrm{sWPHP}}
\newcommand{\rWPHP}{\mathrm{rWPHP}}
\def\GI{\mathrm{GI}}
\def\LI{\mathrm{LI}}
\def\LLI{\mathrm{LLI}}

\newcommand{\APPROX}{\mathrm{APPROX}}

\newcommand{\Comp}{\mathrm{Comp}}

\newcommand{\eb}[2]{\exists #1 \! < \! #2 \,}
\newcommand{\ab}[2]{\forall #1 \! < \! #2 \,}
\newcommand{\ang}[1] { \langle #1 \rangle }
\newcommand\prob{\mathop{\rm Pr}\nolimits}

\newcommand{\bn}{\bar n}
\newcommand{\R}{\mathcal{R}}
\newcommand{\Rpq}{\R_{p,q}}

\let\phi\varphi
\let\epsilon\varepsilon

\newcommand{\ignore}[1]{}

\newcommand{\changed}[1]{#1}

\binoppenalty=\maxdimen
\relpenalty=\maxdimen

\begin{document}

\title{Approximate counting and NP search problems}

\author{Leszek Aleksander Ko{\l}odziejczyk\footnote{Institute of Mathematics, University of Warsaw, \texttt{lak@mimuw.edu.pl}}  ~and Neil~Thapen\footnote{Institute of Mathematics of the Czech Academy of Sciences, \texttt{thapen@math.cas.cz}}}

\maketitle

\begin{abstract}
We study a new class of NP search problems, those which can be proved total 
\changed{using standard combinatorial reasoning 
based on approximate counting. 
Our model for this kind of reasoning is the bounded arithmetic theory  $\APC_2$  of [Je\v{r}\'{a}bek 2009].}
In particular, the Ramsey and weak pigeonhole search problems lie in the \changed{new }
class.
We give a purely computational characterization of this class and show that,
relative to an oracle,
it does not contain the problem CPLS, a strengthening of PLS.

As CPLS is provably total in the theory $T^2_2$, this shows that 
$\APC_2$ does
not prove every $\forall \Sigma^b_1$ sentence which is provable in bounded arithmetic.
This answers the question posed in 
[Buss, Ko{\l}odziejczyk, Thapen 2014] and represents some progress in the 
programme of separating the levels of the bounded arithmetic hierarchy by low-complexity sentences.

Our main technical 
tool is an extension of the ``fixing lemma'' from [Pudl\'{a}k, Thapen 2017], a form of switching lemma,
which we use to show that a random partial oracle from a certain distribution will, with high 
probability, determine an entire computation of a $\P^\NP$ oracle machine.
\changed{The introduction to the paper is intended to make the statements and context of the results accessible 
to someone unfamiliar with NP search problems or with bounded arithmetic.}
\end{abstract}

\section{Introduction} \label{sec:introduction}

An \emph{NP search problem} is specified by a polynomial-time relation
$R(x,y)$ and a polynomial $p(x)$. Given an input $x$, a solution to the problem
is any $y$ such that~$R(x,y)$ holds and~$|y|<p(|x|)$ (where $|x|$ is the length 
of a string $x$). 
We only consider \emph{total} problems, where a solution is guaranteed to exist for all~$x$.
The class of all such problems is called TFNP, standing 
for \emph{total functional}~NP~\cite{TFNP_1991}.

Subclasses of TFNP are sometimes described as consisting
of all search problems which can be proved to be total by some particular 
combinatorial lemma or style of argument~\cite{TFNP_1991, papadimitriou_1994}.
For example, the class PPA ``is based on the lemma that every graph has 
an even number of odd-degree nodes" \cite{beame_et_al}.
Often, the particular lemma or argument can be represented
by a specific axiomatic theory.
In  this paper we study the class, which we call APPROX, of problems 
\changed{that can be proved total using reasoning based on \emph{approximate counting}.
Our model of this kind of reasoning is 
the theory $\APC_2$ developed by Je\v{r}\'{a}bek  in~\mbox{\cite{jerabek_approx,jerabek_hash}},
which}
provides machinery to count the size of a set
well enough to distinguish between sets of size $a$ and $(1+\epsilon)a$,
for $a$ given in binary (but not between sets of size $a$ and~$a+1$),
and to formalize a certain amount of induction in this language.
In this way it can carry out the standard proofs of,
for example, the finite Ramsey theorem 
and the tournament principle~\cite{jerabek_hash}.
\changed{We give a non-logical characterization of APPROX, as 
the class of TFNP problems reducible in a certain sense to a version
of the weak pigeonhole principle for $\P^\NP$ functions.
Examples of problems in the class are natural problems associated with the
finite Ramsey theorem, the usual weak pigeonhole principle, and the ordering principle.}

\changed{Our main result is that, in the relativized setting,
a search problem known as CPLS~\cite{CPLS} is not in APPROX.
Here CPLS is a natural strengthening of 
a complete problem for the class PLS of search problems 
(\emph{polynomial local search}, see Section \ref{subsec:tfnp} below).}

\changed{This work is mainly motivated by an open problem in logic. Our result} 
answers a question about a hierarchy of theories 
collectively known as bounded arithmetic. 
For each $i \in \N$, the theory~$T^i_2$
is axiomatized by induction for formulas at level~$i$ in the polynomial hierarchy.
It is a long-standing open problem 
whether the theories $T^i_2$ can be separated
by sentences expressing that various NP search problems are total --
known as $\forall \Sigma^b_1$ sentences. In other words:
does the class of provably total NP search problems
get strictly bigger as $i$ increases? This is open for $i \ge 2$.

$\APC_2$ lies between~$T^1_2$ and~$T^3_2$.
In~\cite{BKT2014} we pointed out that
the NP search problems typically used in arguments
separating $T^1_2$ from $T^i_2$ for $i \ge 2$
could be proved total using approximate counting.
This led us to state the following open problem, 
which is an important special case
of the more general one: is there any $i$ such 
that $T^i_2$ proves the totality of more 
NP search problems than $\APC_2$ does?

Our result \changed{
in this paper implies that the answer is yes. The totality of $\CPLS$,
which is provable in $T^2_2$,  is not provable in $\APC_2$,
and thus $T^3_2$ proves strictly more NP search problems total than $\APC_2$ does.
This
} makes $\APC_2$ one of the strongest 
natural theories that has been separated from theories 
higher up in the bounded arithmetic hierarchy --
in fact, from $T^i_2$ for the lowest possible~$i$ --
in terms of NP search problems. Intuitively speaking, the
conclusion is that the power of $T^2_2$, $T^3_2, \ldots$ 
to prove many NP search problems total
is based on more than just a limited ability to count. 

\changed{While our motivation is from logic, 
an important part of the methods we use are 
complexity-theoretic,
and may be independently interesting to complexity theorists.}
Our main technical tool is the ``fixing lemma'' from \cite{RRR}.
\changed{This is related to the \emph{switching lemma} of  H{\aa}stad \cite{hastad},
which is used in complexity to separate depth~$d+1$ circuits from depth~$d$ circuits.
The fixing lemma is a simplified,
but more widely applicable, version of  this result.}
It shows that a random partial assignment can, with high 
probability, determine the value of a~CNF. 
We strengthen it slightly, to show that a random partial oracle 
can determine an entire computation of a $\P^\NP$ oracle machine.
The proof of
our version is almost identical, and many definitions are identical, to what appears in~\cite{RRR}.
\changed{In the more technical parts of Sections 4-6} will assume the reader has access to that paper.

\medskip

In the rest of this introductory section we give an overview
of bounded arithmetic and the theory $\APC_2$, describe
the structure of TFNP from this point of view,
and outline how we handle relativization and reductions.
The paper is then structured as follows.

\newcommand{\sct}[1]{\textit{Section~{#1}}.}

\sct{\ref{sec:results}}
We define  $\CPLS$
and our search-problem class
\mbox{APPROX}, formally state our main 
Theorems~\ref{lem:APC_2_characterize} and~\ref{thm:cpls-not-in-ourclass},
obtain some corollaries, 
and give an outline of the proofs. 

\sct{\ref{sec:logic}}  We prove Theorem~\ref{lem:APC_2_characterize},
that the class APPROX
captures the~$\forall \Sigma^b_1$ consequences of~$\APC_2$.
\changed{This section contains the technical work in logic.}

\sct{\ref{sec:fixing}} We \changed{state and} 
prove our version of the fixing lemma.

\sct{\ref{sec:non-reducibility}}  We use the fixing lemma to 
show that a $\P^\NP$ computation is determined by a random partial oracle,
and derive 
Theorem~\ref{thm:cpls-not-in-ourclass}, that CPLS is not in APPROX.

\sct{\ref{sec:propositional}}  We briefly sketch
an alternative way to prove our main result about bounded arithmetic,
\changed{that $\APC_2$ does not prove the totality of CPLS,} by going through propositional
proof complexity rather than NP search problems. 

\sct{\ref{sec:open-problems}}
 We mention some open problems.

%
%
%
%
%

\paragraph{Acknowledgements.}
The first author was partially supported first by grant 2013/09/B/ST1/04390 
and then by grant 2017/27/B/ST1/01951 of the National Science Centre, Poland.
The second author was partially supported \changed{
by GA \v{C}R project 19-05497S}.
Part of this research was carried out during the first author's visit to Prague \changed{in 2017},
funded by  \changed{ERC grant 339691}. 
The Institute of Mathematics of the Czech Academy of Sciences is supported by RVO:67985840.

\changed{We are grateful to Pavel Pudl\'{a}k and Pavel Hub\'{a}\v{c}ek for discussions about this work.}

\subsection{Bounded arithmetic}

Fix a language 
$\LPV$ containing a symbol for every function or relation computed by
a polynomial-time machine. Then \changed{a total} NP search problem 
\changed{can be identified with} 
a true $\LPV$ sentence  of the form
$
\forall x \, \eb{y}{t(x)}R(x,y),
$
where $R$ is a polynomial-time relation, $t$ is a polynomial-time function,
and  $x$ and~$y$  range over natural numbers written in binary notation.
Let $T$ be any sound theory. 
The set of such sentences provable in $T$ then defines a class
of search problems.
For the class to have some reasonable properties, 
$T$ should not be too weak; and to get classes of
the kind usually studied in complexity theory, it should not be too strong.

Natural \changed{examples of suitable} theories $T$ come from bounded arithmetic,
which has close ties to computational complexity. For the purposes of this paper,
we will take such theories to be given by a \emph{base theory} fixing some basic
properties of the symbols of $\LPV$, 
together with one or more axiom schemes that
allow us to do stronger kinds of reasoning, typically induction.
All axioms are universal closures of 
\emph{bounded formulas}, that is, formulas in which all quantifiers 
appear in the form $\ab{x}{t}$ or $\eb{x}{t}$.

In more detail, a \emph{$\PV$ formula} is a quantifier-free formula 
of $\LPV$. 
A \emph{$\Sigma^b_i$ formula} is one of the form
\[
\eb{x_1}{t_1(\bar{z})} \ab{x_2}{t_2(\bar{z}, x_1)} \dots \, \phi(\bar{z}, \bar{x})
\]
where $\phi$ is a $\PV$ formula, the bounds \changed{$t_j$} are $\LPV$-terms, quantifiers
may appear in alternating $\exists$ and $\forall$ blocks, and there are at most $i$ blocks. 
Such formulas define precisely the $\Sigma^p_i$ relations, that is, those at level $i$ in the
polynomial hiearchy. 
The~$\Pi^b_i$ formulas are defined dually. 
The \emph{universal closure} of a formula $\phi(\bar{z})$ is the sentence
$\forall {\bar{z}} \,\phi(\bar{z})$. Given a class of formulas $\Gamma$, we  write
 $\forall\Gamma$ for the set of universal closures of formulas from $\Gamma$.
\changed{Thus, for instance, a $\forall \PV$ sentence states that some polynomial-time
computable property holds for all inputs.}

We will consider two base theories, both containing only $\forall \PV$ sentences.
The first and more usual one is the theory $\PV$, \changed{which 
comes from Cobham's characterization of the polynomial-time functions
as a function algebra \cite{Cook:PV, cobham}.
We will not define $\PV$ here, as the details are not important, but 
it can be thought of as a minimal theory in which all polynomial-time
functions are well-behaved.}

The second, \changed{stronger but simpler theory,} which we denote $\true$, consists 
simply of all $\forall \PV$ sentences
which are true under the standard interpretation in~$\N$. 
This is simpler to understand than $\PV$ \changed{because there is no
list of axioms to keep in mind, and also}
works more naturally for
defining NP search problems. Our results translate 
easily between the two, and a reader unfamiliar with bounded arithmetic
will not go very wrong by reading $\PV$ as $\true$ throughout -- 
see Subsection~\ref{subsec:truth-proof}.

The important family of theories $T^i_2$, for $i \ge 0$, is defined as
\[
T^i_2 := \PV + \Sigma^b_i\textrm{-}\mathrm{IND}
\]
where $\Sigma^b_i$-IND is the usual induction scheme for $\Sigma^b_i$ formulas with parameters.
$\PV$ already proves induction for quantifier-free formulas 
(as that sort of induction can be witnessed by polynomial-time binary search)
so in this setting $T^0_2$ is the same as $\PV$. We write $T_2$ for the union of this family.

We can now state a fundamental theorem. 
By the \emph{$\Sigma^b_i$-definable} functions of a theory we mean the 
\changed{ functions with $\Sigma^b_i$ graphs which the theory proves are total.}

\begin{theorem} [\cite{Buss:bookBA}] \label{the:buss_witnessing}
For $i \ge 0$, the $\Sigma^b_{i+1}$-definable functions of $T^i_2$ are precisely
the $\smash{\P^{\Sigma^p_i}}$ functions, that is, those that are polynomial-time computable 
with an oracle from level $i$ of the polynomial hierarchy.
\end{theorem}

We will also use a related family of theories $S^i_2$, for $i \ge 1$, which is defined by replacing
the $\Sigma^b_i\textrm{-}\mathrm{IND}$ scheme in $T^i_2$ with the apparently weaker scheme
$\Sigma^b_i\textrm{-}\mathrm{LIND}$ in which inductions can only run for polynomially many
steps (in the binary length of a parameter).
We have $T^i_2 \subseteq S^{i+1}_2 \subseteq T^{i+1}_2$~\cite{Buss:bookBA}.
Theorem~\ref{the:buss_witnessing} remains true if $T^i_2$ is replaced by $S^{i+1}_2$ and/or
the base theory $\PV$ is replaced by $\true$.

In addition to being related to computational complexity by Theorem~\ref{the:buss_witnessing},
bounded arithmetic is a natural environment in which to ask 
questions about the provability or consistency of theorems or conjectures from 
complexity theory. For recent examples see~\cite{pich:PCP, mueller_pich}.

We now make the above definitions slightly more complicated.
As in complexity theory,
we typically cannot expect to show that two theories
of bounded arithmetic
are distinct without either making some extra assumption or 
working relative to some oracle. We will use oracles.
We redefine $\LPV$ to include a unary relation symbol 
$\alpha$ standing for ``an arbitrary oracle'', and 
function and relation symbols for all polynomial-time 
machines with oracle access to~$\alpha$. Other formula
classes and theories are redefined to use this extended language.
In particular, $\true$ becomes the set of $\forall \PV$ sentences
which are true in $\ang{\N, A}$ for \emph{every} oracle $A$ 
interpreting the symbol $\alpha$.
Strictly speaking, we should change the names
to $\PV(\alpha)$, $\Sigma^b_i(\alpha)$, $T^i_2(\alpha)$ etc. 
However, since we never use the unrelativized 
versions, we simplify notation by keeping the old names.
The results mentioned above still hold.

\paragraph{An open problem.}
By adapting oracle separation results for the polynomial hierarchy, 
it has been shown that the strength of the (relativized) theories $T^i_2$ increases strictly with $i$:
for each $i$ there is a $\forall \Sigma^b_{i+1}$ sentence provable in $T^{i+1}_2$ but not in $T^i_2$
\cite{bk:boolean}.
A pressing open question in proof complexity\footnote{
This is essentially equivalent to a question in propositional proof complexity about separating 
bounded-depth Frege systems by formulas of fixed depth,
and in particular finding a family of small-width CNF's which have short refutations in bounded-depth Frege but
require long refutations in Res(log).}
is whether 
this remains true 
if we measure the strength of theories only by their $\forall \Sigma^b_k$ consequences for some fixed $k$,
in particular for $k=1$.

We write $\forall \Sigma^b_k(T)$ for the $\forall \Sigma^b_k$ consequences of a theory $T$.
\changed{In particular $\forall \Sigma^b_1(T)$ is a class of total NP search problems (as long as $T$ is sound).}
From \cite{krajicek:counterexample, chiari_krajicek} we know
that 
\[
\forall \Sigma^b_1(\PV) \subsetneq \forall \Sigma^b_1(T^1_2) 
\subsetneq \forall \Sigma^b_1(T^2_2)
\]
and from \cite{chiari_krajicek, ST_search_problems}
we know that for any~$i,k \ge 1$, 
\[
\textrm{if} \ 
\forall \Sigma^b_k(T^i_2) = \forall \Sigma^b_k(T^{i+1}_2)
\ \textrm{then} \ 
\forall \Sigma^b_k(T^i_2) = \forall \Sigma^b_k(T_2).
\]
The following  is open for $k \le 2$:
\begin{equation}\label{q:sep}
\textrm{does \ $\forall \Sigma^b_k(T^2_2) = \forall \Sigma^b_k(T_2)$ ?}
\end{equation}
The answer is expected to be negative, even for $k=0$, 
by analogy with the~$\Pi_1$ separation between I$\Sigma_i$
and I$\Sigma_{i+1}$ given by the second incompleteness theorem.
The case $k=1$ seems to be particularly approachable, as classes
$\forall \Sigma^b_1(T)$ have a natural 
computational interpretation in terms of 
NP search problems.

\paragraph{Approximate counting.} 
Je\v r\'abek~\cite{jerabek_approx, jerabek_hash}
developed a bounded arithmetic theory for approximate counting. Following~\cite{BKT2014} we call this theory%
\footnote{
Our definition is slightly different from Je\v r\'abek's in~\cite{jerabek_hash},
which uses a variant of the surjective weak pigeonhole principle with a smaller difference between domain and range. 
However, the theories prove the same  $\forall \Sigma^b_2$ statements, which is all that matters for this paper.}
$\APC_2$ and define it as $T^1_2$ together with the 
\emph{surjective weak pigeonhole principle} (sWPHP)
for $\P^\NP$ functions, which asserts that no such function can be a surjection from~$n$ to~$2n$, for any~$n>0$.
$\APC_2$ can formalize many arguments in finite combinatorics that use approximate
counting, such as the standard proofs of the finite Ramsey theorem and the tournament principle, as well as some probabilistic reasoning.
It lies between~$T^1_2$ and~$T^3_2$ in strength,
as this instance of the weak pigeonhole principle is provable in $T^3_2$.

In~\cite{BKT2014} we asked the analogue of question (\ref{q:sep})
for $\APC_2$ in place of~$T^2_2$. More specifically:
\begin{equation}\label{q:sep-apc}
\textrm{does \ $\forall \Sigma^b_1(\APC_2) = \forall \Sigma^b_1(T_2)$  ?}
\end{equation}
We expected the answer to be ``no'', but the opposite did not seem completely implausible.
Approximate counting is a powerful tool in finite combinatorics, and typical
combinatorially natural examples of hard $\forall \Sigma^b_1$ statements \changed{that have been}
used to separate $T^1_2$ from $T_2$ 
were known to be provable in $\APC_2$~\cite{BKT2014}.
Moreover it was shown in~\cite{BKZ2015}, by formalizing Toda's theorem, 
that all of bounded arithmetic collapses to the analogue
of $\APC_2$ if we add a parity quantifier to the language.  

Both \cite{BKT2014} and later \cite{atserias_thapen} showed
 unprovability results for various natural  subtheories of $\APC_2$, 
but these fell well short of answering (\ref{q:sep-apc}).
In fact, they 
were obtained using a $\forall\Sigma^b_1$ sentence that is actually provable in $\APC_2$.

\subsection{TFNP}\label{subsec:tfnp}

\changed{As already discussed,} in our language
a total NP search problem is simply  a 
true $\forall \Sigma^b_1$ sentence, 
that is, one of the form
$\forall x \, \eb{y}{t(x)} R(x,y)$
where $R(x,y)$ is a $\PV$ formula and $t$ is \changed{an} $\LPV$-term.
This represents the search-task of finding a witness $y$, given $x$. We will often assume that the bound
$y<t(x)$ is implicit in $R(x,y)$, 
and we will usually write $R(x,y)$ or just $R$
as a name for the search problem.

As before, polynomial-time is defined relative to an
oracle symbol $\alpha$, 
and we will occasionally use notation like $R(x,y; \alpha)$ 
to emphasize the specific oracle being used.
The oracle leads to a slight complication in what we mean when we say a 
search problem is total: $\forall x \, \eb{y}{t(x)} R(x,y; \alpha)$ must be true in $\ang{\N, A}$ for \changed{\emph{every}} oracle $A$ interpreting $\alpha$. 
This behaviour is essentially
the same as what is called a total type-2 NP search problem in e.g.~\cite{beame_et_al, buss-johnson}.

We define two important notions of reducibility between search problems~$Q$ and~$R$.
We will introduce one more in Subsection \ref{subsec:apc-tfnp}.

\begin{definition} \label{def:many_one_reduction}
$Q(x,y)$ is \emph{polynomial-time many-one reducible}, or simply \emph{reducible},
to $R(x',y')$, written $Q \le R$,
 if there are polynomial-time functions $f$ and~$g$ and a polynomial-time relation $P$ 
(all of which may query the oracle $\alpha$) such that
\[
R(f(x), y'; P(x,\cdot)) \rightarrow Q(x, g(x,y'); \alpha) 
\]
holds for all $x$, $y'$ and $\alpha$, where $P(x, \cdot)$
represents the oracle $\{z:P(x,z)\}$.  
Two problems are \emph{equivalent} if they are reducible to one another.
\end{definition}

\begin{definition}
$Q(x,y)$ is \emph{polynomial-time Turing-reducible} to $R(x',y')$ if there is a polynomial-time relation $P$ and
a polynomial-time oracle machine $M$ which, on input $x$, makes a series of (adaptive) queries
to $R(x',y';P(\ang{x,x'},\cdot))$. If all replies are correct,
then $M$ outputs  some $y$ such that $Q(x,y;\alpha)$. 
\end{definition}

We are interested in search problem classes corresponding to bounded
arithmetic theories,
\changed{in the sense that the class captures the search problems
proved total by the theory (but see 
Section~\ref{subsec:truth-proof} below).}
 There has been a research programme,
motivated partly by the logical separation question discussed above,
to characterize these classes. 

\begin{itemize}[leftmargin=18pt]
\item 
$\PV$ corresponds to $\mathrm{FP}$, the class of search problems which can be solved
in deterministic polynomial time~\cite{Cook:PV, Buss:bookBA}.
\item 
$T^1_2$ corresponds to $\PLS$~\cite{PLS, bk:boolean}. 
A $\PLS$ problem is given by polynomial-time
\emph{neighbourhood} and \emph{cost} functions $N_x$ and $C_x$
and \emph{domain} predicate~$F_x$, 
such that 
$0 \in F_x$ and if $y \in F_x$, then $|y| \le |x|^k$ for some fixed $k$.
A solution to an instance $x$ is any~$y \in F_x$ such that either $N_x(y) \notin F_x$
or $C_x(N_x(y)) \ge C_x(y)$. Such a $y$ exists because costs cannot decrease indefinitely.
A complete problem for the class is to find a local minimum 
for a function on a bounded-degree graph.
\item 
$T^2_2$ corresponds to $\CPLS$~\cite{CPLS}, a generalization of $\PLS$ described below.
\item 
For $k \ge 1$, $T^k_2$ corresponds to a class $\GI_k$ defined by
the $k$-turn game induction principle~\cite{ST_search_problems}
(see also \cite{BeckmannBuss:localSearch, BeckmannBuss:definableSearch}).
\changed{Roughly speaking, a complete problem for $\GI_k$ is:
given a sequence of $k$-round 2-player games, winning strategies
for opposite players in the first and last games, and reductions between
neighbouring games in the sequence, find an error in one of the strategies
or one of the reductions.}
Equivalent search problems include further generalizations of $\PLS$
and principles about feasible Nash equilibria~\cite{PudlakThapen:alternating}, 
and~$\LLI_k$, the $k$-round linear local improvement principle~\cite{KNT:localimprove}.
\end{itemize}

The theory $T^i_2$ is equivalent to a natural formalization of 
``every $\P^{\Sigma^p_i}$ machine has a computation on every input'',
essentially by Theorem~\ref{the:buss_witnessing}.
The search problems above can thus be thought of as the
 projections onto TFNP 
of increasingly strong computation models. 
This can be taken further: there are
two  ``second-order'' bounded arithmetic
theories, $U^1_2$ and $V^1_2$, 
which are equivalent (with respect to their~$\forall \Sigma^b_1$ consequences)
to similar statements about computations of, respectively, PSPACE and EXPTIME 
machines~\cite{Buss:bookBA, KNT:localimprove}.
In terms of NP search problems, 
from~\cite{KNT:localimprove, BeckmannBuss:localImprove} we have:
\begin{itemize}[leftmargin=18pt]
\item 
$U^1_2$ corresponds to 
$\LLI_{\log}$, the linear local improvement 
principle with polynomially many rounds.
\footnote{
Unfortunately $\LLI_{\log}$ is rather complicated to describe or use. 
The authors believe that it is equivalent to the simpler game induction
principle $\GI$, which is like the principle $\GI_k$ 
of~\cite{ST_search_problems} but with polynomially many rounds.}
\item
$V^1_2$ corresponds to  $\LI$, the local improvement principle. 
\end{itemize}
One can think of the classes 
$\mathrm{FP} \subseteq \GI_1 \subseteq \GI_2  \subseteq \dots  \subseteq \LLI_{\log}  \subseteq \LI$ 
as forming a backbone for TFNP,
\changed{arising natural from the hierarchy of bounded
arithmetic theories or of computation models -- see Figure~\ref{fig:TFNP}}.
This  could be extended even beyond bounded arithmetic, 
to say $\forall \Sigma^b_1(\text{PA})$ or 
$\forall \Sigma^b_1(\text{ZFC})$~\cite{beckmann:peano, tabatabai:flows},
or by using potentially stronger systems of reasoning, as in~\cite{goldberg_papadimitriou}.

\begin{figure}
\centering

\begin{tikzpicture}

\node at (1,5) {PPA};
\node at (3,5) {PPP};
\node at (2,4) {PPAD};
\draw [thick, ->] (4.6,5.8) -- (1.35,5.2);
\draw [thick, ->] (4.8,5.8) -- (3.35,5.2);
\draw [thick, ->] (1.2,4.75) -- (1.8,4.25);
\draw [thick, ->] (2.8,4.75) -- (2.2,4.25);
\draw [thick, ->] (2, 3.75) to [out=-90,in=180] (4.65, 0);

\node at (5,0) {FP};
\node at (5,1) {PLS $\equiv \GI_1$};
\node at (5,2) {CPLS $\equiv \GI_2$};
\node at (5,3) {$\GI_3$};
\node at (5,4) {$\GI_4$};
\node at (5,6) {$\LLI_{\log}$};
\node at (5,7) {$\LI$};

\draw [thick, ->] (5,0.75) -- (5,0.25);
\draw [thick, ->] (5,1.75) -- (5,1.25);
\draw [thick, dotted, ->] (5,2.75) -- (5,2.25);
\draw [thick, dotted, ->] (5,3.75) -- (5,3.25);
\draw [thick, dotted, ->] (5,4.65) -- (5,4.25);
\draw [thick, ->] (5,5.75) -- (5,5.25);
\draw [thick, dotted, ->] (5,6.75) -- (5,6.25);

\node at (7, 1.1) {PWPP};
\node at (8.35,1.1) {HOP};
\node at (10,1.1) {RAMSEY};
\draw [thick, ->] (7, 0.85) to [out=-100,in=0] (5.35, 0);
\draw [thick, ->] (8.25, 0.85) to [out=-125,in=0] (5.35, 0);
\draw [thick, ->] (9.5, 0.85) to [out=-140,in=0] (5.35, 0);
\draw [thick, ->] (6.05, 1.8) to (6.8, 1.35);
\draw [thick, ->] (6.05, 1.9) to (8.1, 1.35);

\node at(8.5, 2.5) {APPROX};
\draw [thick, ->] (8.3, 2.25) to (7.1, 1.35);
\draw [thick, dotted, ->] (8.7, 2.25) to (8.35, 1.35);
\draw [thick, dotted, ->] (9.1, 2.25) to (9.7, 1.35);
\draw [thick, ->] (5.35, 2.95) to (7.8, 2.7);
\draw [thick, ->] (7.9, 2.25) to (5.65, 1.3);


\end{tikzpicture}
\caption{
{A diagram showing some inclusions between selected classes in TFNP, 
including our  results. Some classes are named by their
complete problems. Solid arrows are strict inclusions, relative to some oracle. 
Dotted arrows are inclusions not known to be strict. For separations among PPA, PPP, PPAD, PLS 
see~\cite{beame_et_al, morioka, bureshop_morioka, buss-johnson}.
For other references see Sections~\ref{subsec:tfnp} and~\ref{subsec:apc-tfnp}.
CPLS separates $\GI_3$ from APPROX and also separates $\GI_2$ from HOP and PWPP.
WEAKPIGEON separates APPROX from~PLS.
HOP separates APPROX from PWPP. 
Solid arrows from PPP to HOP~\cite{BKT2014} and from PPP to PWPP have been omitted for readability.
}}
\label{fig:TFNP}
\end{figure}

However, experience suggests\footnote{
A related issue in propositional proof complexity is the difficulty 
of finding  candidates for separating
the Frege and extended Frege proof systems~\cite{ABBCI}.} 
that it is difficult to find any natural ``combinatorial'' NP search problem
that is not already provably total in~$U^1_2$, and thus reducible to~$\LLI_{\log}$.
In particular, $U^1_2$ is strong enough to formalize the counting arguments
needed to prove the totality of complete problems for the well-known classes PPA and PPP
introduced in~\cite{papadimitriou_1994}. 
Thus, all problems in those classes are reducible to~$\LLI_{\log}$.

On the other hand the bijective pigeonhole principle, 
called OntoPIGEON in the search problem literature, 
is a complete problem for the class PPAD~\cite{bureshop_morioka}
which is contained in both PPA and PPP.
Standard proof-complexity lower bound arguments for the pigeonhole principle~\cite{KPW:PHP, PBI:PHP}
show that this problem is not provably total in any theory $T^k_2$ and not reducible to any $\GI_k$.

Finally let us mention the search problem class PWPP~\cite{jerabek_roots},
based on the injective weak pigeonhole principle \changed{search problem WEAKPIGEON}, which
is contained in $\GI_2$ and PPP but not in PLS~\cite{mpw:wphp, krajicek:counterexample}. 
We will discuss \changed{this and} two other  search problems, RAMSEY and HOP,
in Section \ref{sec:results}.

It is open whether the hierarchy 
$\GI_2 \subseteq \GI_3 \subseteq \dots$ is strict. 
This is essentially the same problem as the separation of $\forall \Sigma^b_1(T^{i+1}_2)$ from
$\forall \Sigma^b_1(T^{i}_2)$ discussed above.

\subsection{True and provable reductions}\label{subsec:truth-proof}

In the previous section, we did not explicitly say what it means for a 
search problem class to correspond to a theory $T$. 
The obvious meaning, that the class is precisely $\forall \Sigma^b_1(T)$,
potentially has a problem. Namely, such a class does not have the desirable property of
being closed under many-one reductions, unless the reductions
work provably in $T$. There may even be two $\PV$ formulas
$R_1$ and $R_2$ which ``semantically'' define the same relation on~$\N$,
and thus the same search problem by the usual complexity-theoretic definition,
but are such that~$T$ proves that one is total but not the other.

There are two natural ways around this issue. One is to define our class
as the closure of $\forall \Sigma^b_1(T)$ under many-one reductions. The other is to
stick to theories $T$ that contain the set $\true$ of all true $\forall \PV$ sentences,
and exploit the fact that the statement that a reduction works is such a sentence.

The next lemma shows that,
for theories of the kind we consider, these two approaches have the same result.
In this paper we prefer the second one.

\begin{lemma} [folklore, see also \cite{hanika}]\label{lem:folklore}
Let $Q(x,y)$ be an NP search problem.
Let~$T$ be a bounded arithmetic theory containing $\PV$
and with axioms closed under substituting polynomial-time relations for oracles. 
Then (1) and (2) below are equivalent. 
If~$T$ contains $S^1_2$, then (3) is also equivalent. 
\begin{enumerate}
\item
$Q$ is provably total in $T+ \true$.
\item
$Q \le R$ for some TFNP problem $R$ provably total in $T$.
\item 
$Q$ is Turing reducible to a TFNP problem $R$ provably total in $T$.
\end{enumerate}
\end{lemma}

\ignore{
\textbf{[Cut or move somewhere else]}
{\it [In  light of Lemma~\ref{lem:folklore}, we will focus on proving results
for theories including $\true$. 
This seems more natural for dealing with TFNP classes
and lets us avoid all technicalities related to the definition of $\PV$.
Moreover, our main result, which is an unprovability/non-reducibility statement,
is actually stronger in the presence of $\true$.
All the other results could be reformulated
as (slightly more cumbersome) 
statements about standard bounded arithmetic theories
in a routine way.]} }

\begin{proof}
Suppose (1) holds. We have that
$T + \forall z \, \phi(z) \vdash \forall x\,  \exists y \, Q(x,y)$
for some $\PV$ formula $\phi$ such that $\forall z \, \phi(z) \in \true$. Hence
$T \vdash \, \forall x \, [\exists z \, \neg \phi(z) \vee \exists y \, Q(x,y)]$.
Since $T$ is a bounded theory, by Parikh's theorem~\cite{parikh} we may add some
term $t(x)$ bounding both existential quantifiers.
Therefore $T \vdash \forall x \, \eb{y}{t(x)} R(x,y)$
where $R(x,y)$ is the formula $\neg \phi(y) \vee Q(x,y)$.
Now~$R$ is an NP search problem, provably total in~$T$, and
$Q$ is reducible to $R$ in $\N$ using the identity function, 
since~$\phi(y)$ is true for every $y$ and every oracle. Hence (2) holds.

Now suppose (2) holds. Then (3) is immediate. 
For (1), from the definition of a reduction,
 there are $\PV$ function symbols $f,g$ and a relation symbol $P$ such that, for every oracle $A$,
\[
\ang{\N,A} \vDash \forall x \, \forall y'\, [R(f(x), y'; P(x, \cdot)) \rightarrow Q(x, g(x,y'); A)].
\]
Now define the search problem $R^*(x', y'; \alpha) := R(x', y'; P(x, \cdot))$.
Then $\true$ proves 
$\forall x' \, \exists y' \, R^*(x',y') \rightarrow \forall x \, \exists y \, Q(x,y)$,
and by the property of closure under subsitution for the oracle, $T$
proves that $R^*$ is total. Hence we have (1). 

Lastly we show that (3) implies (1) under the stronger assumption.
Turing reducibility means that there is a polynomial-time oracle machine $M$
which, on input $x$, makes oracle queries to $R$ and, if the replies
are correct, outputs $y$ such that $Q(x,y)$.
Formally, for every $A$, $\ang{\N,A} \vDash \forall x\, \forall w \, \phi(x,w)$
for a $\PV$ relation  $\phi(x,w)$ expressing that: if 
$w$ is a computation of $M$ on input~$x$, and every oracle query $x'$
in $w$ has a reply $y'$ in $w$ such that $P(x',y')$, then 
$Q(x, \mathrm{output}(w))$.
But~$T$ proves that, for all $x$, such a $w$ exists, 
since $\Sigma^b_1$-LIND is enough to construct~$w$ query by query.
Hence $T + \true \vdash \forall x\, \exists y\, Q(x,y)$.
\end{proof}

\section{Main definitions and results} \label{sec:results}

\subsection{Coloured polynomial local search}

We study a search problem introduced in \cite{CPLS}.
We will need several results about it from~\cite{RRR}, so we take the definition verbatim from there.

Let $a,b,c$ be parameters.
Consider a levelled directed graph
whose nodes consist of all pairs $(i,x)$ 
from~$[0,a) \times [0,b)$.
We refer to~$(i,x)$ as \emph{node~$x$ on level~$i$}. If $i<a-1$, this node has a single
neighbour in the graph, node~$f_i(x)$  on level $i+1$. Every node in the graph is coloured with 
some set of colours from~$[0,c)$.
The principle CPLS, \emph{coloured polynomial local search}, 
says that the following three statements cannot all be true:
\begin{enumerate}[(i)]
\item
Node $0$ on level $0$ has no colours.
\item
For every node $x$ on every level $i < a-1$, and for every colour $y$, 
if the neighbour $f_i(x)$ of $x$ on level~\mbox{$i+1$} has colour $y$, 
then $x$ also has colour $y$.
\item
Every node $x$ on the bottom level $a-1$ has at least one colour, $u(x)$.
\end{enumerate}
When the parameters $a,b,c$ are universally quantified, CPLS is expressed as a $\forall \Sigma^b_1$
sentence about an oracle $\alpha$ encoding the functions $f_i$ and $u$ and a 
predicate~$G$, where $G_i(x,y)$ means ``node~$x$ on level~$i$ has colour~$y$''. 

To describe it explicitly as a search problem: the inputs are the parameters~$a,b,c$
and a solution is a witness that one of items (i)-(iii) above fails. That is,
a colour $y$ such that $G_0(0,y)$;
or a node $(i,x)$ and a colour $y$ such that $G_{i+1}(f_i(x), y) \wedge \neg G_i(x,y)$;
or a node $(a-1,x)$ such that $\neg G_{a-1}(x, u(x))$.

To see that the principle is true, or equivalently that the search problem is total,
suppose that (i)-(iii) hold simultaneously.
Then we can reach a contradiction by arguing inductively on $i$
that for all $i$, some node on level $i$ has no colours. This argument can be formalized as a
proof of CPLS in $T^2_2$. Moreover, this has a kind of converse, in that it is shown in~\cite{CPLS} that 
CPLS is complete for the search-problem class $\forall\Sigma^b_1(T^2_2)$
 with respect to many-one reductions.
 
\changed{It is worth pointing out that CPLS is a generalization of PLS in the sense
that it simplifies to a \changed{PLS-complete problem} 
if the parameters are restricted in a certain way,
for instance if we fix the number of colours~$c$ to~1. 
In that situation, given $a,b$, we define the domain $F$ of the PLS problem by putting $(i,x) \in F$ if and only if $\neg G_i(x,0)$ holds,
that is, if $x$ does not have the unique possible 
colour~$0$ on level~$i$. 
For $(i,x) \in F$, let $N(i,x) := (i+1,f_i(x))$ and $C(i,x) := a - i$.}

\changed{Thus, one way of looking at CPLS is that 
the parameter $a$ is a bound on possible costs,
$b$ is a bound on the number of potential neighbours
 of a given element,
and $c$ is a bound on the number of potential reasons why an element
could fail to be in the domain. 
It can be shown that CPLS becomes a PLS problem 
whenever one of the three parameters is constrained to be
is at most polylogarithmic in the maximum of the two others.}

\subsection{Retraction WPHP and $\Sigma^p_2$ search problems}\label{subsec:apc-tfnp}

The \emph{retraction weak pigeonhole principle}~\cite{jerabek_hash}
asserts that, for $n >0$, given two functions 
$f\colon n \rightarrow 2n$ and $g \colon 2n \rightarrow n$
there must be some $v<2n$ such that $f(g(v)) \neq v$.
It is true, because otherwise simultaneously
$f$ would be a surjection and $g$ an injection. If $f$ and $g$ are polynomial time, 
this principle naturally gives rise to a problem in TFNP. 
We will be in a situation
where $f$ and $g$ are $\P^\NP$, and for this we will define a more complex kind of search problem.

\begin{definition}
A \emph{$\Sigma^p_2$ search problem}
is specified by a $\mathrm{coNP}$ relation $R(x,y)$ and a polynomial bound $q$
such that $\forall x \,\eb{y}{2^{q(|x|)}}\, R(x,y)$. We will often assume that
the bound $q$ is implicit in $R$ and will not write it.
The problem represents the search-task of finding such a $y$, given $x$.
\end{definition}

As this definition makes sense outside the context of bounded arithmetic, 
we have written it in standard complexity-theory notation.
But we could alternatively define a $\Sigma^p_2$ search problem 
as a true $\forall \Sigma^b_2$ sentence, in the style of
our syntactical definition of TFNP problems.

\changed{
A basic example of such a search problem is: for a fixed $\P^\NP$ machine~$M$,
given an input, find a computation of $M$ on this input. 
Here we assume
 that a computation includes
witnesses for all NP queries that get the answer YES, so that the property of being a 
(correct) computation is thus $\mathrm{coNP}$.
We specify precisely what we mean, as it will be important
in what follows.}

\begin{definition} \label{def:P_NP_computation}
We define  a $\Pi^b_1$ formula 
\emph{``$w$ is a computation of $M$ on input~$v$''}.
The formula interprets $w$ as a sequence $\bar{q}, \bar{r}, \bar{y}$
of respectively NP queries, YES/NO replies and witnesses to replies. It expresses that
\begin{enumerate}
\item\label{def:P_NP-part1}
For each $i$, $q_i$ is the $i$-th query asked by $M$ in a computation on input~$v$, assuming the previous replies were
$r_1,\ldots, r_{i-1}$,
\item\label{def:P_NP-part2}
For each $i$, if reply $r_i$ is YES then $y_i$ witnesses this,
\item\label{def:P_NP-part3}
For all sequences $\bar{z}$ of possible counterexamples, for each $i$, if reply $r_i$ is NO then
$z_i$ is not  a counterexample to this. 
\end{enumerate}
The machine only accesses the oracle $\alpha$ via
the NP queries. We say that $w$ is a \emph{precomputation of $M$ on input $v$}
if it satisfies \ref{def:P_NP-part1}.~and~\ref{def:P_NP-part2}.~above.
\end{definition}

Note that being a precomputation of $M$ on a given input is a $\PV$ formula, so it makes sense to speak 
of precomputations also when $\alpha$ is only partially defined (as long as the defined part is
large enough to verify~2. above). Note also that it is implicit in clause \ref{def:P_NP-part1}.~of the definition that each query asked in a computation
of $M$ depends only on the input and the previous YES/NO replies to queries, not on the witnesses to the previous replies.

\changed{We now give our main definitions.}

\begin{definition}
$\rWPHP_2$ is a class of $\Sigma^p_2$ search problems.
A problem in the class is specified by  $\P^\NP$ machines \changed{computing} functions
$f_x(u)$ and $g_x(v)$, where we treat one argument $x$
as a parameter. The functions $f_x$ and $g_x$ are constrained 
to take values less than $2x$ and $x$ respectively.  
An input to the problem is a size parameter
$x$. A solution is a pair \changed{$\ang{v,w}$} 
such that $v<2x$,
$w$ is a computation of $f_x(g_x(v))$ in the sense of Definition~\ref{def:P_NP_computation},
and the output of $w$ is not~$v$.
\end{definition}

\begin{definition} \label{def:PLS_counterexample}
An NP search problem $Q(x,y)$ is \emph{$\PLS$ counterexample reducible} 
to a $\Sigma^p_2$ search problem $R(x',y')$ if 
there is a $\PLS$ problem $P(x'', y'')$ and polynomial time
functions $d$ and $e$ with the following property:
for any~$x, y', y''$ such that $P(\ang{x, y'}, y'')$,
either
\begin{enumerate}
\item
$d(x, y'')$ witnesses that $R(e(x), y')$ is false, or
\item
$Q(x, d(x, y''))$.
\end{enumerate}
\changed{
As in Definition~\ref{def:many_one_reduction}, the oracle called by $R$ 
is allowed to be a polynomial-time variant of the oracle $\alpha$ called by $Q$.
Precisely, there is a polynomial-time relation~$A$ querying $\alpha$ such that,
 in the description  above, $R$  queries
$A(x, \cdot)$ as its oracle
rather than $\alpha$.
} 
\end{definition}

\begin{definition}
The search problem class $\APPROX$ consists of all NP search problems which are
$\PLS$ counterexample reducible to an $\rWPHP_2$ problem.
\end{definition}

\changed{The definition of PLS counterexample reducibility}
should be understood as follows. We are given $x$
and want to find $y$ such that $Q(x,y)$. We create an input $e(x)$
to $R$ and are given a purported solution $y'$ for which it is 
claimed that $R(e(x), y')$ --  
\changed{since $R$ is a $\Sigma^p_2$ search problem,} this is a coNP claim which we
cannot check directly. We then use $\ang{x, y'}$ as input for our
PLS problem $P$, and find a solution $y''$. Then either~1. or~2. above
holds, that is, either the coNP claim about $R$ was false and $d(x, y'')$
is a counterexample, or~$d(x, y'')$ is a solution to our original problem.

\changed{
As far as we know, this notion of reducibility has not been studied before.
We give some examples. Unfortunately, in the examples
we know which are relatively simple, one part or
another of the definition becomes trivial.
}

\changed{
Firstly, every PLS problem $Q(x,y)$ is PLS counterexample reducible
to the trivial $\Sigma^p_2$ search problem defined by letting
$R(x',y')$ hold for all $y' < 2^{|x'|}$. 
In the reduction, both $d$ and $e$ are the identity mapping
and $P(\ang{x, y'}, y'')$ holds exactly if $Q(x,y'')$ does.
}

\changed{
Now consider four search problems:
\begin{enumerate}
\item
The $\Sigma^p_2$ problem TOURNAMENT: 
given $x$ and a binary oracle $\alpha$ representing a tournament on the set of vertices $[0,x)$,
find $y$ coding a set of at most $\lceil \log x \rceil$ vertices
which dominates the tournament, that is, coding vertices
$y_1, \dots, y_k$ such that for all $z<x$ we have $\alpha(y_i, z)$ for some~$i$.
\item
The TFNP problem HOP (the \emph{Herbrandized
ordering principle}\footnote{
\emph{Herbrandized} refers to the presence of the predecessor function~$h$
which  turns the task of finding the $\preccurlyeq$-smallest 
element of $[0,x)$, which is naturally a $\Sigma^p_2$ problem, 
into a TFNP problem. The \emph{generalized iteration principle} of~\cite{chiari_krajicek}
is similar in spirit, as is the \emph{graph ordering principle} in the propositional proof complexity literature.
}~\cite{BKT2014}): 
given $x$ and oracles for a binary relation $\preccurlyeq$ and a unary function~$h$, 
find either a witness that $\preccurlyeq$ restricted to $[0,x)$ is not a total ordering,
or a witness that $h$ is not the $\preccurlyeq$-immediate predecessor function on $[0,x)$.
\item
The TFNP problem RAMSEY:
given $x$ and an oracle $R$ for a graph on~$[0,x)$,
find $y$ encoding a set $s \subseteq [0,x)$ of cardinality $\lfloor\log x/2 \rfloor$ such that~$[s]^2$ is homogeneous with respect to~$R$.
\item
The TFNP problem WEAKPIGEON\footnote{This is different in a non-essential way
from how this problem is defined in~\cite{jerabek_roots}, where it takes as input a circuit
for the function $g$.}:
given $x$ and an oracle $g$ for a function from $[0,2x)$ to $[0,x)$,
find $v,v'<2x$ such that $v \neq v$ and $g(v) = g(v')$.
The class of problems that are reducible to WEAKPIGEON was given the name PWPP in~\cite{jerabek_roots}.
\end{enumerate}
}

\changed{
HOP is reducible to TOURNAMENT in the following sense. 
We are given a size $x$ and a binary 
relation $\preccurlyeq$ on $[0,x)$
for which we want to solve HOP.
Consider $\preccurlyeq$ as a tournament on $[0,x)$
and give it as input to TOURNAMENT. Suppose
$s=\{y_1, \dots, y_k\}$ is a purported solution to 
TOURNAMENT, with $k$ polylogarithmic in~$x$.
By polynomial-time search we can find either
a witness in $s$ that $\preccurlyeq$ is not a total ordering
(thus solving HOP),
or a $\preccurlyeq$-least element $y_i$ of $s$. 
In the second case, compute $y':=h(y_i)$.
If $y' \succcurlyeq y_i$, then $y_i$
is a solution to HOP. Otherwise, comparing $y'$ to each $y_j \in s$
will either give us a witness that $\preccurlyeq$ is not a total ordering or 
reveal that $y'$ witnesses that $s$ is not a correct solution to TOURNAMENT.
}

\changed{
The above is a description of a PLS counterexample reduction of HOP to 
TOURNAMENT. The function $e$ translating input to HOP into input to 
TOURNAMENT is the identity. We then compute from $s$ either a solution
to HOP or a witness that $s$ is not a solution to TOURNAMENT.
However we can do this computation  in polynomial time, while 
Definition~\ref{def:PLS_counterexample} more generally allows it to be done by a call to a PLS problem $P$
(assisted by a polynomial time ``decoding'' function~$d$).
}

\changed{Our final examples of PLS counterexample reducibility are related
to the class $\APPROX$.
We will show in Theorem~\ref{lem:APC_2_characterize} below that this class coincides with the class of
NP search problems that are provably total using approximate counting. This has the following consequence.
\begin{corollary}[of Theorem~\ref{lem:APC_2_characterize}]
\label{cor:things_in_APPROX}
The NP search problems $\mathrm{HOP}$, $\mathrm{RAMSEY}$ and $\mathrm{WEAKPIGEON}$ are in $\APPROX$.
Therefore they are $\PLS$ counterexample reducible to $\rWPHP_2$ problems.
\end{corollary}
\begin{proof}
The problems $\mathrm{RAMSEY}$~\cite{pudlak_ramsey, jerabek_hash},
$\mathrm{HOP}$~\cite{BKT2014} 
and WEAKPIGEON are all provably total in $\APC_2$.
(This is also true for a stronger version of HOP 
in which $\preceq$ is only required to be a partial ordering, not a total ordering~\cite{BKT2014}.)
\end{proof}}

\changed{We are only aware of a simple direct proof of the reduction in the case of WEAKPIGEON. 
Suppose we are given a polynomial-time function $g$ and a parameter $x$ as input. 
To solve WEAKPIGEON we want to find a collision in~$g$, which we interpret as a function from $2x$ to $x$.
Let a function $f$ be given by a $\P^\NP$ machine inverting~$g$, as follows. On input $u<x$,
$f$ queries whether $\eb{v}{2x} ( g(v)=u)$. If the answer is NO, $f$ gives up and outputs~$0$.
If the answer is YES, $f$ outputs the maximal such~$v$, found by binary search.
Consider the problem in $\rWPHP_2$ specified by $f$, $g$ and the parameter~$x$. 
Recall that a solution is a pair $\ang{v,w}$ such that $v<2x$,
$w$ is a computation of $f(g(v))$ in the sense of Definition~\ref{def:P_NP_computation},
and the output of $w$ is not~$v$.}

\changed{In the PLS counterexample reduction of WEAKPIGEON to this $\rWPHP_2$ problem,
the auxiliary functions $d$ and $e$ will be the identity.
The reduction procedure $P$ will once again be polynomial-time. 
To describe it, suppose we are given $v$ and a precomputation $w$
of $f$ on input $u = g(v)$, with output $v' \neq v$ 
(we can ignore the computation of $g(v)$, since $g$ is polynomial time). We want to either
find a collision in $g$, or a witness that $w$ is not 
a computation of $f$, that is, that some NO reply
recorded in $w$ is wrong.
}

\changed{
First suppose that $v'=0$ and was output by $f$ because the reply
to the first query ``$\eb{z}{2x} ( g(z)=u)$?" was NO. Then 
$v$ is a witness that this reply is wrong. 
Otherwise, $v'$ was found by binary search. In this case, we 
follow~$w$ until the first place where the binary search interval excludes $v$.
Then, if the interval is strictly below $v$, we can use $v$ to witness that 
the most recent NO reply was wrong. If it is above $v$, then this has to be the result of a YES answer, with a witness $v''>v$ recorded in $w$. Thus we have found a collision, since $g(v'')=g(v)=u$.
}

\medskip

\changed{
Finally let us record here a fact which appears in Figure~\ref{fig:TFNP}.
\begin{proposition}
Relative to an oracle, $\mathrm{HOP}$ is not reducible to $\mathrm{WEAKPIGEON}$
and therefore is not in $\mathrm{PWPP}$.
\end{proposition}
This is implicit in the proof
in~\cite{BKT2014}
that $T^1_2 + $iWPHP does not prove HOP is total. 
It explicitly follows from~\cite{mueller_riis}, which generalizes this
result and works over the stronger base theory
$\forall\PV(\N)$.}

\subsection{Results}\label{subsec:results}


\changed{Recall that  $\APPROX$ was  defined in the last
subsection as the set of 
NP search problems which are
$\PLS$ counterexample reducible to an $\rWPHP_2$ problem.}

\begin{theorem} \label{lem:APC_2_characterize}
$\forall \Sigma^b_1(\APC_2 + \true)$ = $\APPROX$.
\end{theorem}

\changed{In other words, APPROX captures the class of TFNP problems which
are provably total using approximate counting.}
This is proved in Lemmas~\ref{lem:APC_to_counterexample} and~\ref{lem:counterexample_to_APC}
in Section~\ref{sec:logic}, by applying standard witnessing techniques 
\changed{from bounded arithmetic} to the definition of the theory $\APC_2$.

\begin{theorem}\label{thm:cpls-not-in-ourclass}
$\CPLS$ is not in $\APPROX$.
\end{theorem}

This is proved in Section~\ref{sec:non-reducibility}, using a lemma
about random oracles proved in Section~\ref{sec:fixing}.
We briefly sketch the proof.
We first fix an alleged $\PLS$ counterexample reduction of $\CPLS$
to a problem from $\rWPHP_2$ specified by a pair of $\P^{\NP}$ functions 
$f$ and $g$, 
then choose a large size parameter $n$
and use it to set suitable values for the parameters $a,b,c$ of $\CPLS$.
We define a notion of a ``legal'' partial oracle, which in particular is one
which does not contain any witness to $\CPLS$.
We adapt a lemma on random restrictions from~\cite{RRR} to show
that with high probability a random partial oracle $\rho$ from a certain 
distribution will ``fix'' YES or NO replies to all $\NP$ queries made in a 
$\P^{\NP}$ computation,
in the sense that these replies will never become wrong in any legal 
extension of~$\rho$ (Lemma~\ref{lem:fixed_computations}).
It follows that most partial oracles $\rho$ from this distribution 
will fix computations of $(f \circ g)(v)$ in this sense
on most inputs $v$. This is enough for $\rho$ to determine a solution \changed{$\ang{v,w}$}
to our instance of $\rWPHP_2$ for which it is difficult to find a
counterexample (Lemma~\ref{lem:good-rho}).
Finally we again adapt a proof from~\cite{RRR} to show, by an Adversary argument
in which the Adversary's strategy
uses only legal extensions of~$\rho$, 
that our PLS reduction
is not able to find a witness to CPLS from $\ang{v,w}$.

\changed{Our} main result about bounded arithmetic, answering the question 
posed in~\cite{BKT2014}, is an immediate consequence of Theorems~\ref{lem:APC_2_characterize} 
and~\ref{thm:cpls-not-in-ourclass}:

\begin{corollary} \label{cor:CPLS_not_APC2}
The principle $\CPLS$ is not provable in $\APC_2$. Since it is provable in $T^2_2$, 
it follows that, in the relativized setting, 
$\APC_2$ does not prove all $\forall \Sigma^b_1$ consequences of full bounded arithmetic $T_2$.
\end{corollary}

This naturally also limits the strength of theories that are provable in $\APC_2$,
such as the following one based on the \changed{usual (non-Herbrandized)} ordering principle.

\begin{corollary}
Consider the theory consisting of $T^1_2$ together with axioms stating
that for every $\PV$ formula $R(x,y)$ and every $a$, 
if $R$ is a partial ordering on~$[0,a)$ then $[0,a)$ contains an $R$-minimal element.
This theory, in the relativized setting, is strictly weaker than $T^2_2$.
\end{corollary}

\begin{proof} 
This theory is provable in $T^2_2$ by straightforward induction on $a$.
It is also provable in $\APC_2$ by an entirely different proof \changed{involving a
reduction to the tournament principle}, 
as is shown
in~\cite{BKT2014}
(by an argument due to Je\v{r}\'{a}bek). 
By Corollary~\ref{cor:CPLS_not_APC2}, the theory does not prove~$\CPLS$,
\changed{hence is weaker than $T^2_2$.}
\end{proof}

Both CPLS and the ordering principle have short proofs in the \emph{resolution}
propositional proof system, and the argument above could also be used to show
that the ordering principle is not ``complete" for resolution, in the sense that there
are things with short proofs in resolution which do not follow from it.
But it is not clear what the most suitable notion of ``follow from" is here.
\changed{
\begin{corollary} \label{cor:ramsey_HOP}
$\CPLS$ is not polynomial-time Turing reducible to
$\mathrm{RAMSEY}$ or~$\mathrm{HOP}$.
\end{corollary}
\begin{proof}
 $\mathrm{RAMSEY}$ and $\mathrm{HOP}$
are in $\APPROX$ by 
Corollary~\ref{cor:things_in_APPROX}.
If $\CPLS$ were polynomial-time Turing reducible to either of these problems,
then Lemma~\ref{lem:folklore} would imply that $\CPLS$ is provable in $\APC_2 +\true$,
contradicting Theorem~\ref{lem:APC_2_characterize} and Theorem~\ref{thm:cpls-not-in-ourclass}.
\end{proof}
}

\section{Witnessing and definability} \label{sec:logic}

This section contains our main technical work in logic: a proof of Theorem~\ref{lem:APC_2_characterize}
via two lemmas corresponding to the two containments in the statement of the theorem.
The proofs assume some familiarity with bounded arithmetic.

Intuitively, $\APC_2$ is a combination of $T^1_2$ and the weak pigeonhole principle,
and what we show is that the NP search problems provably total in $\APC_2$
arise as a combination of PLS (which is known to correspond to~$T^1_2$~\cite{bk:boolean}) and the weak pigeonhole principle, with an important difference that, while a proof 
can make many ``calls'' to WPHP, our reductions only allow one call.
\changed{
We also introduce the following technical condition on reductions,
which we will need in our non-reducibility proof in Section~\ref{sec:non-reducibility}.
\begin{definition}\label{def:clean_reducibility}
We say that a NP search problem $Q$ is \emph{cleanly} PLS counterexample reducible
to a $\Sigma^p_2$ search problem $R$, if
$Q$ is PLS counterexample reducible to $R$ as in Definition~\ref{def:PLS_counterexample}
with the extra condition that the function $e$, which produces inputs to $R$
from inputs to $Q$, does not make any oracle calls.
\end{definition}
It may be helpful to think of such an $e$ as a translation between size parameters, for
which the structure of the oracle does not matter.
}

\begin{lemma} \label{lem:APC_to_counterexample}
Every $\NP$ search problem provably total in $\APC_2 + \true$ is 
\changed{cleanly} $\PLS$ counterexample reducible to an $\rWPHP_2$ problem.
\end{lemma}

\begin{proof}
Let $Q(x,y)$ be an $\NP$ search problem. Assume that 
\[\APC_2 + \forall z\,\phi(z) \vdash \forall x\,\exists y\, Q(x,y),\] 
where $\phi(z)$ is a $\PV$ formula such that $\N \vDash \forall z\,\phi(z)$ for all oracles.
Thus we have 
\[\APC_2 \vdash \exists y\, Q(x,y) \lor \exists z\,\neg\phi(z). \]
Writing out the definition of $\APC_2$, this means
\[
T^1_2 + \forall b, c \, \eb{v}{2b} \ab{u}{b} \, e(c,u) \neq v 
\vdash \exists y\, Q(x,y) \lor \exists z\,\neg\phi(z)
\]
where the formula on the left is $\sWPHP$
for  a universal $\P^\NP$ machine $e(c,u)$ running code $c$ on input $u$ with time bound $|c|$.
\changed{The quantifier $\forall b, \dots$ 
should strictly speaking be $\forall b \! \neq \! 0, \dots$
but we suppress this here and below
for the sake of readability.} 
Replacing~$T^1_2$ with the stronger theory $S^2_2$ and 
moving $\sWPHP$ to the right hand side gives
\changed{\begin{equation}\label{eqn:pre-amplification}
S^2_2 \vdash [\eb{b, c}{s'(x)} \, \ab{v}{2b} \eb{u}{b} \, e(c,u) = v ] \lor 
\exists y\, Q(x,y) \lor \exists z\,\neg\phi(z),
\end{equation}
where we have also used Parikh's theorem~\cite{parikh}
to bound $b$ and $c$ by some term~$s'(x)$.
We may assume $s'(x)$ has the form $2^{|x|^k} \! + \! x$ for some $k \in \N$,
where ``$+x$'' is included  to make it possible to 
recover~$x$ from~$s'(x)$ as required below.}
\changed{
The formula in square brackets asserts that $\sWPHP$ fails for $e$ 
for parameters $b,c$ less than $s'$. By a standard technical fact stated and proved below as Lemma~\ref{lem:WPHP_amplify},
we can amplify this failure to a $\P^\NP$ function $F$ which is a surjection from $(s')^5$ to $2(s')^5$, with no parameters.
This lets us remove one block of existential quantifiers from~(\ref{eqn:pre-amplification})
to give, setting the term $s$ to be $(s')^5$,
\begin{equation}\label{eqn:buss-witnessing-1}
S^2_2 \vdash
[ \ab{v}{2s} \eb{u}{s} \, F(u) = v  ]
\lor \exists y\, Q(x,y) \lor \exists z\,\neg\phi(z).
\end{equation}
}
To match the definition of $\rWPHP_2$,
we define a $\P^\NP$ function of two arguments~$a,u$ by $f_{a}(u) = \min(F(u), 2a-1)$.
Then (\ref{eqn:buss-witnessing-1}) is equivalent to
\begin{equation}\label{eqn:buss-witnessing-2}
S^2_2 \vdash
[ \ab{v}{2s} \eb{u}{s} \, f_s(u) = v  ]
\lor \exists y\, Q(x,y) \lor \exists z\,\neg\phi(z).
\end{equation}
The sentence in (\ref{eqn:buss-witnessing-2}) is $\forall \Sigma^b_2$, so
by Buss' witnessing theorem for $S^2_2$~(\cite{Buss:bookBA})
 there is a $\P^\NP$ machine which, provably in $T^1_2$,
maps the input parameters $x, v$ to a triple $\ang{u,y,z}$ witnessing one of the three existential quantifiers. 
Let~$g$ be defined so that $g_{s(x)}(v)$ first computes $x$ from $s(x)$ and then
outputs the first component $u$ of this witnessing function applied to $\ang{x,v}$, as long as
$u < s$; 
otherwise, $g$ outputs $0$. We have
\[T^1_2 \vdash
[ \ab{v}{2s} \, f_{s}(g_s(v)) = v  ]
\lor \exists y\, Q(x,y) \lor \exists z\,\neg\phi(z).
 \]
Now, $f_s(g_s(v)) = v$ can be written as a $\Pi^b_2$ formula: 
\[\forall w\, [w \textrm{ is a computation of } f_s(g_s(v)) \rightarrow \mathop{\mathrm{output}}(w) = v],\] 
where $w$ is suitably bounded by a term in $x$, and 
``$w$ is a computation of $f_s(g_s(v))$'' is a $\Pi^b_1$ formula 
as in Definition~\ref{def:P_NP_computation}, describing the $\P^\NP$ 
machine that first computes $g$ and then computes $f$ on the output.
%

So we have 
\begin{multline*}
T^1_2 \vdash
\ab{v}{2s} \forall w\, [w \textrm{ is not a computation of } f_s(g_s(v)) \lor \mathop{\mathrm{output}}(w) = v ]   \\
\lor \exists y\, Q(x,y) \lor \exists z\,\neg\phi(z).
\end{multline*}
The formula in square brackets is now $\Sigma^b_1$, so by the PLS witnessing theorem for $T^1_2$ (\cite{bk:boolean})
there is a PLS problem $P(x'', y'')$ witnessing this whole sentence.
That is, if we solve $P$ on input $x'' = \ang{x,v,w}$ and 
find $y''$ such that $P(x'',y'')$, then one of the following holds:
\begin{enumerate}
\item
$w$ is not a precomputation of $f_s(g_s(v))$, or has output $v$\changed{,}
\item
$y''$ is a tuple containing a witness that some NO reply
in $w$ is wrong\changed{,}
\item
$y''$ is a tuple containing a witness to $\exists y\, Q(x,y)$.
\end{enumerate}
We know that~$y''$ cannot contain a witness to the last disjunct
$\exists z\, \neg \phi(z)$ as
by assumption $\N \vDash \forall z\,\phi(z)$.

Using the notation of Definition~\ref{def:PLS_counterexample},
letting $d$ be the function that outputs the witness of incorrectness in case~2., 
and the witness to $\exists y\, Q(x,y)$ in case~3., 
and setting $e(x) = s(x)$, we see that $Q$ is \changed{cleanly}
PLS counterexample reducible to the $\rWPHP_2$ problem given by $f$ and $g$.
\end{proof}

\changed{
To complete the proof of Lemma~\ref{lem:APC_to_counterexample} we need a technical lemma from~\cite[Section~2]{thapen_WPHP_models}
about ``amplifying'' failures of WPHP.}

\begin{lemma} \label{lem:WPHP_amplify}
\changed{Let $e$ be a $\P^\NP$ function. There is a $\P^\NP$ function
$F$ such that provably in~$S^2_2$, 
if $e(c, \cdot)$ is a surjection  $b\rightarrow 2b$ then $F(a,b,c, \cdot)$ is a surjection
$b \rightarrow 2^{|a|}b$.
Furthermore there is a $\P^\NP$ function $F'$ such 
that provably in $S^2_2$, if $e(c, \cdot)$ is a surjection $b \rightarrow 2b$ for some $0<b,c<s$,
then $F'$ is a surjection $s^5 \rightarrow 2s^5$ (with no parameters).}
\end{lemma}

\begin{proof}
\changed{Let $h(a,b,c,u)$ be the function that, assuming $u < 2^{|a|}b$,
interprets $u$ as a pair $\ang{u_0,u_1}$ with $u_0 < b, u_1 < 2^{|a|}$, 
and outputs the number $e(c,u_0)2^{|a|} + u_1$.
If $e(c,\cdot)$ is a surjection from $b$ onto $2b$, then
$h(a,b,c,\cdot)$ is a surjection from $2^{|a|}b$ onto~$2^{|a|+1}b$.
Consider now the function $F(a,b,c,u)$ which, 
 given $u < b$, applies the composition of $e(c,\cdot), h(1,b,c,\cdot),
h(2,b,c,\cdot), h(4,b,c,\cdot),\ldots, h(2^{|a|-1},b,c,\cdot)$, to $u$, in the order shown.
If $e(c,\cdot)$ is a surjection from $b$ onto $2b$, then
$F(a,b,c,\cdot)$ is a surjection from $b$ onto $2^{|a|}b$.
To prove this, we consider any fixed number $v<2^{|a|}b$
and show by reverse induction on $i < {|a|-1}$
that $v$ can be obtained by applying the composition of
$h(2^i,b,c,\cdot), h(2^{i+1},b,c,\cdot),\ldots, h(2^{|a|-1},b,c,\cdot)$
to some argument $u$ below $2^ib$. This is where we need the scheme $\Sigma^b_2$-LIND,
which is available in $S^2_2$.}

\changed{For the last claim, we note that if we code quadruples as $\ang{\ang{\cdot,\cdot},\ang{\cdot,\cdot}}$
where $\ang{\cdot,\cdot}$ is Cantor's pairing function, then for each $k$ larger than a fixed natural number
any quadruple of numbers less than $k$ has code less than $k^5$.
We treat the argument $x<s^5$ of $F'$ 
as a quadruple $\ang{a',b',c', u'}$ and set $F'(\ang{a',b',c',u'}) := F((a'+1)^6, b', c', u')$.
Since such arguments $x$ will cover all quadruples of numbers below $s$,
it follows from the properties of $F$ that $F'$ contains every number less than $2^{6|s|}b$
in its range, and in particular every number less than $2s^5$.}
\end{proof}

\changed{
The next lemma is the converse to Lemma~\ref{lem:APC_to_counterexample}.}

\begin{lemma}  \label{lem:counterexample_to_APC}
If $Q$ is an $\NP$ search problem 
$\PLS$ counterexample reducible to an $\rWPHP_2$ problem, 
then $Q$ is provably total in $\APC_2 + \true$.
\end{lemma}

\begin{proof}
Let $Q$ be an $\NP$ search problem that is $\PLS$ counterexample reducible to the $\rWPHP_2$ problem given by the functions $f$ and $g$.
Let $P(x'',y''), d, e$ be as in the definition of $\PLS$ counterexample reducibility. 
Consider the $\PV$ formula $\xi(x,v,w,y,y'')$ defined by 
\begin{multline*}
v < 2e(x) \land P(\ang{x,\ang{v,w}}, y'') \land y = d(x,y'') \\ \land [ y \textrm{ does not witness that } w \textrm{ is not a computation of } f_{e}(g_{e}(v)) \neq v],
\end{multline*}
with $y''$ suitably bounded. 
We view $\xi$ as defining an NP search problem with input $x$ 
and output $\ang{v,w,y,y''}$. Notice that $\xi(x,v,w,y,y'')$ implies
that $Q(x,y)$, so we have $Q \le \xi$.

We claim that 
$\APC_2$ 
proves that $\xi$ is total.
To see this, work in $\APC_2$ and consider some input $x$. 
By $\sWPHP(\PV_2)$, there is some $v<2e(x)$ which is outside of the range of $f_{e(x)}$ on inputs below~$e(x)$. 
By $T^1_2$, there is some computation of $f_{e(x)}(g_{e(x)}(v))$, say $w$, which by the choice of $v$ must produce an output different from $v$. 
Again by $T^1_2$, there is a solution to $P$ on input $\ang{x,\ang{v,w}}$, say $y''$. 
Clearly, $y = d(x,y'')$ cannot witness that~$w$ is not a computation of $f_{e(x)}(g_{e(x)}(v))$ 
with output different from $v$, so $\xi(x,v,w,y,y'')$ holds. This proves the claim.

We have shown that $Q$ is reducible to a problem provably total in $\APC_2$.
It follows from Lemma \ref{lem:folklore} that $Q$ is provably total in $\APC_2 + \true$.
\end{proof}

\changed{
Lemmas~\ref{lem:APC_to_counterexample} and~\ref{lem:counterexample_to_APC} have the following additional consequence.
\begin{corollary} \label{cor:clean} 
The class $\APPROX$ can equivalently be defined as  the set of $\NP$ search problems 
which are \emph{cleanly} $\PLS$ counterexample reducible to an $\rWPHP_2$ problem.
\end{corollary}
}

\section{Fixing lemma} \label{sec:fixing}

This section contains our main technical result in complexity, 
Lemma~\ref{lem:fix_tree},
which is an extension of the ``fixing lemma'' from~\cite{RRR}.
There, the fixing lemma is a limited switching lemma which says the following:
given suitable parameters $a,b,c$ for CPLS, 
for a well-chosen probability distribution on partial restrictions to an oracle $\alpha$
encoding $(f_i)_{i<a-1}$, $u$, $(G_i)_{i<a}$, a random restriction has a relatively high probability
of determining the value of a narrow CNF in propositional variables standing for bits of $\alpha$.
Importantly,
the restriction does not reveal a witness to CPLS; in particular,
the (unsatisfiable) CNF asserting that there is no witness to CPLS
has to be determined to be true.

In our application in the proof of Theorem \ref{thm:cpls-not-in-ourclass},
we want to fix answers to the NP queries made in a $\P^\NP$ computation.
Each query is (the negation of) a CNF, but now there are many of them,
and they are made adaptively depending on earlier replies. So we cannot use
the lemma from~\cite{RRR} directly.
Instead we adapt the proof to show that given a low-depth decision 
tree labelled with CNFs,
 with high probability a random restriction fixes
the truth values of all CNFs along some branch.
This is the basic content of Lemma~\ref{lem:fix_tree} below.

Our definitions are essentially as the same as in~\cite{RRR}, and so is one \changed{of the proofs}. 
We will repeat some definitions almost verbatim, but will only give high-level descriptions
of some other definitions and of the unchanged proof details.

We think of the bits of the oracle as propositional variables. So, for example, for each 
node $(i,x)$ there are
$\log b$ variables $(f_i(x))_0, \dots, (f_i(x))_{\log b-1}$ expressing the value
of $f_i(x)$.
A total oracle is defined
by a total assignment to all variables. We will be working with partial oracles,
which we will also call \emph{partial assignments} or \emph{restrictions}.

We copy in full the definition of a \emph{random restriction} from~\cite{RRR}. 
First, a \emph{path} in a partial assignment
$\beta$ is a maximal sequence 
$(i, x_0), \dots, (i+k, x_k)$ of nodes such that  
$f_{i+j}(x_j) = x_{j+1}$ in $\beta$ for each $j \in [0,k)$.
\changed{Note that a path may consist of a single node $(i,x)$
if $x$ is neither in the domain of $f_i$ nor in the range of $f_{i-1}$.} 
If all functions $f_i$ are partial injections, 
then every node is on some unique path.

\begin{definition} (\cite[Definition 5.7]{RRR})
\label{def:CPLS_restriction}
Fix parameters $0<p,q<1$. 
Let $\mathcal{R}_{p,q}$ be the distribution of random restrictions
chosen as follows. 
\begin{enumerate}
\item[R1.]
For each pair $i<a$ and $x<b$, with probability 
$(1-p)$ include $(i,x)$ in a set $Z$. For each $i<a$, choose $f_i$ uniformly at random from the partial injections
from the domain $\{ x<b : (i,x) \in Z \}$
 into $b$. 
\item[R2.]
Set colours on the path beginning at $(0,0)$ so that $G_i(x,y)=0$ for 
all~$y$ for all nodes~$(i,x)$ on that path. 
\item[R3.]
For every other path $\pi$, with probability $(1-q)$
\emph{colour} $\pi$ randomly with one colour. That is, choose uniformly at random a colour $y$ and, for every node $(i,x)$ on $\pi$, set $G_i(x,y)=1$ and then set $G_i(x,y')=0$ for all $y'\neq y$.
\item[R4.]
Finally consider each node $(a-1, x)$ on the bottom level. It is on some 
path $\pi$. If~$\pi$ was coloured at step R3, then set $u(x)=y$ where $y$ is the unique colour assigned to $\pi$ (that is, $G_{a-1}(x,y)=1$). 
Otherwise leave~$u(x)$ undefined.
\end{enumerate}
We will also use $\mathcal{R}_{p,q}$ to denote the support of this distribution.
\end{definition}

We take the definitions of \emph{legal restrictions}
and \emph{good restrictions} from~\cite[Definition 5.4 and Lemma 5.8]{RRR}. 
\emph{Legal} restrictions are those that meet a minimal standard of ``niceness" -- 
on every path either no colour variables are set, 
or they are all set in one of a few particular ways which do not immediately witness CPLS. 
\changed{Among restrictions $\rho \in \R_{p,q}$, the only ones that are not legal are those
that contain a path connecting $(0,0)$ with some node of the form $(a-1, x)$.}
Our lower bound in the next section
will make use of a game played between a Prover, who is trying 
to witness CPLS
by making oracle queries, and an Adversary 
who is trying to answer queries in a way that does not witness CPLS. 
It will turn out that the Adversary can restrict herself
to answers that come from legal restrictions.
In effect, we do not have to worry about 
the evaluation of formulas under restrictions which are not legal.

A \emph{good} restriction is a legal one 
which is of typical size, measured in various ways -- in particular
no path is very long, and there is a reasonable fraction
of variables unset at every level. A \emph{bad} restriction is one which is not good.

It may be useful to keep in mind what the analogous definitions would be if we were dealing with 
the more familiar pigeonhole principle PHP instead of CPLS. A legal restriction would be any restriction
representing a partial injection. With probability parameter $p$, a random restriction would choose holes 
independently with probability $1-p$, and then randomly map some pigeons to the chosen holes. A good
restriction would be a legal one that leaves at least, say, a fraction $p/2$ of holes unset.

We choose a suitable large $n$ and fix our parameters as
$a=b=n$, $c=\lfloor n^{1/7}\rfloor$, $p=n^{-4/7}$ and $q=n^{-2/7}$,
where $b$ and $c$ are powers of $2$.

\begin{lemma} \label{lem:prob_bad} \emph{(\cite[Lemma 5.8]{RRR})}
The probability that a random restriction is bad is exponentially small in $n$.
\end{lemma}

\begin{definition}\label{def:fixing}
Let $\rho$ be a restriction. We say that a CNF $B$ is:
\begin{itemize}
\item \emph{fixed to $0$} by $\rho$ if $\rho$ falsifies $B$, that is, if for some conjunct in $B$ each literal in the
conjunct is set to $0$ by $\rho$,
\item \emph{fixed to $1$} by $\rho$ if it is not fixed to $0$ by any legal extension of $\rho$.
\end{itemize} 
\end{definition}

Note that a legal restriction can fix a CNF to at most one truth value. The following proposition
is therefore obvious.

\begin{proposition} \label{lem:preserve_fixing}
If $\rho$ fixes a CNF $B$ then every 
extension of $\rho$ also fixes~$B$ to the same value.
\end{proposition}

It follows from the proof of the
 ``fixing lemma''  \cite[Lemma~5.9]{RRR} that, for~$k$ reasonably small compared to $n$,
the probability that a given $k$-CNF is 
fixed by a random restriction is relatively high -- in fact, the probability that it
is \emph{not} fixed is $O(kn^{-1/7})$. 
We need a slightly more general version that also bounds some conditional probabilities.

\begin{lemma} [conditional fixing lemma] \label{lem:conditional_fixing}
Let $A_1, \dots, A_m$ be a sequence of $k$-CNFs and
$e_1, \dots, e_m$ be a sequence of $0/1$ values
such that
\[
\prob [ \rho \textrm{~is bad} \, | \, 
\textrm{$\rho$ fixes each $A_i$ to $e_i$}] < 1/2.
\]
Let $B$ be a $k$-CNF.
Then
\[
\prob [ \rho \textrm{~does not fix~} B \, | \, 
\textrm{$\rho$ is good and $\rho$ fixes each $A_i$ to $e_i$}] < 12kn^{-1/7}.
\]
\end{lemma}

\begin{proof}
If we remove the CNFs $A_i$, this is essentially the ``fixing lemma" of \cite[Lemma~5.9]{RRR} and our proof is almost identical.
Define 
\begin{gather*}
F = \{ \rho \in \mathcal{R}_{p,q} :
\textrm{$\rho$ fixes each $A_i$ to $e_i$} \} \\
G = \{ \rho \in \mathcal{R}_{p,q} :
\textrm{$\rho$ is good} \} \\
S = \{ \rho \in F \cap G : \rho \textrm{~does not fix~} B \} 
\end{gather*}
so that our $S$ is the intersection with $F$ of the set $S$ defined in \cite{RRR}. 
The assumption gives us that $\prob[F \cap G] / \prob[F] \ge 1/2$ and
our goal is to show that $\prob[S]/\prob[F \cap G]$ is small.

Every $\rho \in S$ does not falsify $B$ but does have 
some legal extension which falsifies $B$.
Exactly as in \cite{RRR} we define a function $\theta$ on $S$ by $\theta(\rho) = \sigma'$,
where~$\sigma'$ is a certain minimal legal extension of $\rho$. We have, over $\rho \in S$,
\begin{enumerate}
\item
$\prob[\theta(\rho)] / \prob[\rho] \ge \frac{1}{2} n^{1/7}$ 
\item
$\theta$ is at most $3k$-to-one
\item
$\theta(\rho) \in F$ (although it may happen that $\theta(\rho) \notin F \cap G$).
\end{enumerate}
Items 1 and 2 are proved as in \cite{RRR}. Item 3 is immediate from Proposition~\ref{lem:preserve_fixing}.

Now partition $S$ as $S_0, \dots, S_{3k-1}$ where
$S_i = \{ \rho \in S : \rho$ is the $i$th preimage of $\theta(\rho) \}$.
Then 
\begin{align*}
\prob [S_i]
=
\sum_{\rho \in S_i} \prob[ \rho ] 
= 
\sum_{\rho \in S_i} \prob[ \theta(\rho) ]  \frac{\prob[\rho]}{\prob[\theta(\rho)]}
&\leq
2n^{-1/7}  \sum_{\rho \in S_i} \prob[ \theta(\rho) ] \\
&\leq 
2n^{-1/7} \prob[F]
\end{align*}
where for the last inequality we use that $\sum_{\rho \in S_i} \prob[ \theta(\rho) ] \leq \prob[F]$, 
since $\theta$ is an injection from $S_i$ to $F$.
This step is the main difference from \cite{RRR}, 
which uses only that $\theta$ is an injection
from $S_i$ to $\mathcal{R}_{p,q}$, giving the weaker bound 
$\sum_{\rho \in S_i} \prob[ \theta(\rho) ]
 \leq \prob[\mathcal{R}_{p,q}] = 1$.

It follows that 
$\prob[S] 
= \prob[S_0] + \dots + \prob[S_{3k-1}]
\leq 6kn^{-1/7} \prob[F]$.
Hence
 $\prob[S]/\prob[F \cap G] \leq 12kn^{-1/7}$ as required.
\end{proof}


\begin{lemma} \label{lem:fix_tree}
Consider a complete binary decision tree in which
each internal node is labelled with a $k$-CNF 
and has outgoing edges for NO and YES answers.
A node $z$ and a restriction $\rho$ are 
\emph{compatible} if $\rho$ is good and, for every CNF~$B$ on the path down
from the root to $z$, $\rho$ fixes $B$ to the value specified by the outgoing edge along the path.

Let $\epsilon = \prob [\rho$ is bad$]$. 
A node $z$ is \emph{big}
if \mbox{$\prob[\rho$ is compatible with $z]>\epsilon$}.
Let $S_d$ be the set of good  restrictions $\rho$
which are compatible with some big node at depth $d$.
Then 
\[
\prob[S_d] \ge
1-d  \cdot 12kn^{-1/7} - 2^{d+1} \epsilon.
\]
\end{lemma}

This will be used in the next section, where the decision tree
will model a computation of a $\P^\NP$ machine. In particular $d$
and $k$ will be polylogarithmic in $n$ and $\epsilon$ will be exponentially
small in $n$. It follows from the lemma that at least one node on the
bottom level, and thus at least one computation of the machine,
 is compatible with some $\rho$.

\begin{proof}
We use induction on $d$.
For the base case $d=0$, first observe that every good restriction
is compatible with the root. It follows
that the root is big, as \changed{thanks to Lemma \ref{lem:prob_bad}} we may assume that $\epsilon < 1/2$.
Hence $S_0$ is just the set of good restrictions.

At depth $d$ in the tree, by the definition of compatibility
each restriction in~$S_d$ is compatible with exactly one big node. 
Consider any such big node~$z$.
It is labelled with a $k$-CNF $B$ and has a NO
child $z_0$ and a YES child~$z_1$. Define~$P_z$ as
\[
\prob[\textrm{$\rho$ is not compatible with either $z_0$ or $z_1$
$|$ $\rho$ is compatible with $z$}]
\]
This is equal to the probability that $\rho$ does not fix $B$, under
the condition that $\rho$ is good and \changed{$\rho \in C_z$, where $C_z$ is the
set of restrictions which correctly fix all CNFs above $z$.
Notice that $\prob[\textrm{$\rho$ is bad}]=\epsilon$,
so in particular $\prob[\rho$~is bad and $\rho \in C_z] \le \epsilon$.
On the other hand
$\prob[\rho$ is good and $\rho \in C_z] > \epsilon$, since $z$ is big.
Hence of restrictions $\rho \in C_z$, the fraction which are good is at least $1/2$.
This is the condition we need to apply Lemma~\ref{lem:conditional_fixing}, which
gives~$P_z < 12kn^{-1/7}$. }


Hence, summing over big nodes at level $d$, the probability that $\rho$ is compatible with some (not necessarily big) 
node at depth $d+1$
is at least
\[
\sum_{z \in \{0,1\}^d, \textrm{\! $z$\! big}} (1-P_z) \prob[\textrm{$\rho$ is compatible with $z$}]
\ge (1-12kn^{-1/7}) \prob[S_d].
\]
To obtain $S_{d+1}$ we must finally remove the restrictions which are compatible with 
non-big nodes
at depth $d+1$. But there are at most $2^{d+1}$ such nodes, so the probability of being compatible
with any of them is at most $2^{d+1} \epsilon$. A straightforward calculation shows that
\[(1-12kn^{-1/7})(1-d  \cdot 12kn^{-1/7} - 2^{d+1} \epsilon) - 2^{d+1} \epsilon \ge 1-(d+1) \cdot 12kn^{-1/7} - 2^{d+2} \epsilon ,\]
which completes the inductive step.
\end{proof}

\section{Non-reducibility} \label{sec:non-reducibility}

Consider a restriction~$\rho$ and a $\Sigma^b_1$ formula
$\eb{y}{t} \theta(a,y)$, where $\theta$ is a $\PV$ formula
and $a$ is some number. We say that this formula \emph{is witnessed} in  
$\rho$ if there is some $b<t$ such that $\theta(a,b)$ holds
in $\rho$. That is, if you run the computation verifying $\theta(a,b)$
and answer queries to $\alpha$ with values from $\rho$,
these values are all defined and the computation is accepting.

Recall \changed{from Definition \ref{def:P_NP_computation}}
that a precomputation of a $\P^\NP$ machine 
contains a correct witness for every YES reply, but may be wrong
about NO replies.

\begin{definition}
Let $\rho$ be a restriction.
A precomputation $w$ of a $\P^\NP$ machine~$M$ is \emph{fixed} by $\rho$
if both of the following hold.
\begin{enumerate}
\item
For every $\NP$ query in $w$ with a YES reply, the witness 
 provided by $w$ is correct in $\rho$.
\item
No $\NP$ query  in $w$ with a NO reply is witnessed in any legal $\sigma \supseteq \rho$.
\end{enumerate}
We  say that $\rho$ \emph{fixes a precomputation of $M$ on input $v$}
if there is some such~$w$.
For a function $f$ computed by a $\P^{\NP}$ machine,
we write $\rho \Vdash w \colon f(x)=y$ if $\rho$ fixes a precomputation
$w$ of $f$ on input $x$ that outputs $y$, and we write
$\rho \Vdash f(x)=y$ if $\rho$ fixes some such $w$.
\end{definition}


If $w$ is fixed by $\rho$ then $\rho$ fixes, in the sense of Definition \ref{def:fixing}, 
each DNF representing an $\NP$ query made in $w$. Note that $w$ does not
have to be a computation of $M$ relative to any complete oracle $\alpha$ extending $\rho$. 
\changed{In fact, if $\rho$ is legal and $w$ contains a query like ``is there a witness to CPLS?'', 
then $w$ \emph{cannot} be a computation of $M$ relative to any complete $\alpha$, 
because $\alpha$ will contain a witness to a YES answer, but the answer given in $w$ will be NO.}

The symbol $\Vdash$ is intentionally chosen to be the same one as in forcing. In fact, 
one could formulate the concept of fixing in terms of a forcing relation, with the restrictions as forcing conditions.
However, attempting to preserve all the trappings of forcing in the context of finite combinatorics leads to some
annoying issues, so in this paper we do not explore this possibility further.

\begin{lemma} \label{lem:single_valued}
For a $\P^{\NP}$ function $f$, a restriction $\rho$ and an input $x$,
there is at most one $y$ such that $\rho \Vdash f(x)=y$.
\end{lemma}

\begin{proof}
The progress of a $\P^{\NP}$ precomputation depends only on the 
YES/NO replies to $\NP$ queries, not on the witnesses chosen.
In all precomputations of $f(x)$ fixed by $\rho$ these replies are necessarily the same.
\end{proof}

Below a ``suitable" $n$ is one for which $n^{1/7}$ is a power of two.

\begin{lemma} \label{lem:fixed_computations}
Let $M$ be a $\P^\NP$ machine, running on inputs $x$ with $|x|$ polylogarithmic in $n$.
For all suitable large \changed{enough} $n$, for every such input $x$,
\[
\prob_{\rho \sim \Rpq}[\textrm{$\rho$ fixes a precomputation of $M$ on $x$}]
\ge 1-n^{-1/6}.
\]
\end{lemma}

\begin{proof}
We can model a run of $M$ on $v$
as a decision tree $\mathcal{T}_M$. The height~$d$ of $\mathcal{T}_M$ is bounded by the running time of $M$. 
At each node the tree makes an $\NP$ query; by negating the reply, we can view
this as a query to a \mbox{$k$-CNF}, where $k$ is some obvious syntactic upper bound 
on the time needed to verify a witness to the query. Since $M$ is a $\P^\NP$ machine,
$k$ can be chosen polynomial in the running time of $M$.
So we can apply Lemma~\ref{lem:fix_tree} with $k=d=|n|^c$ for some~$c \in \N$. 
This gives the lower bound
\begin{equation}\label{eq:tree-bound}
1-|n|^{2c} n^{-1/7} - 2^{|n|^c+1} \epsilon
\end{equation}
on the probability that $\rho$ is compatible with
one of the leaves of $\mathcal{T}_M$. 
By Lemma~\ref{lem:prob_bad} the probability $\epsilon$ that 
$\rho$ is bad is exponentially small in~$n$, so the bound in (\ref{eq:tree-bound}) 
is at least $1-n^{-1/6}$ for $n$ sufficiently large. 

Finally, suppose $\rho$ is compatible with a leaf of $\mathcal{T}_M$.
We form a precomputation $w$ by answering queries with the replies
given on the path from the root to the leaf. For each YES reply, it follows
from the definition of fixing a DNF to~1 (that is, fixing~a CNF to 0) that $\rho$ provides
enough information to verify at least one witness to the reply; 
we make some such witness part of $w$.
\end{proof}

\begin{lemma}\label{lem:good-rho}
Let a search problem in $\rWPHP_2$ be given by $\P^\NP$
functions $f_x(u)$ and $g_x(v)$. 
Let \changed{$s(n)$} be quasipolynomial in $n$.
Then for all suitable large \changed{enough}~$n$, 
\[
\prob_{\rho \sim \Rpq}[\textrm{there exist $v \neq v'$ such that $\rho \Vdash f_s(g_s(v)) = v'$}] \ge 1-3n^{-1/6}.
\]
\end{lemma}

\begin{proof}
Choose $n$ sufficiently large. 
Let $M$ be the $\P^\NP$ machine which takes input~$n, v$ and computes 
$f_s(g_s(v))$ by first computing
$g$ and then $f$. As $s$ is quasipolynomial in $n$, we may assume
that $M$ satisfies the assumption of 
Lemma~\ref{lem:fixed_computations} on input size. We will write just
$f$ and $g$ below, suppressing the parameter $s$.

Consider the machine $M$ running on inputs $v<2s$.
By Lemma~\ref{lem:fixed_computations}, 
for any fixed~$v$, a random $\rho$ fixes 
a precomputation of $M$ on $v$ with probability at least $1-n^{-1/6}$.
It follows that with probability at least $1-3n^{-1/6}$
a random~$\rho$ simultaneously fixes precomputations for at least $2/3$ of
all inputs $v$ (as otherwise the fraction of pairs $(\rho, v)$ in which $\rho$ does not fix a precomputation on $v$ would be more than $n^{-1/6}$). 
Fix such a $\rho$.
 In particular, there are at least $s+1$ many distinct
inputs $v_0, \dots, v_s$ for which~$\rho$ fixes precomputations 
$w_0, \dots, w_s$. 

The machine $M$ first computes $u=g(v)$, 
which is necessarily less than~$s$, and then computes $f(u)$.
Hence, by the pigeonhole principle, 
there is some~$u$ for which there exist distinct $i,j$ such that 
$\rho \Vdash w_i \colon g(v_i)=u$
and $\rho \Vdash w_j \colon g(v_j)=u$.
(Here and below we are abusing our notation slightly, as~$w_i$ and~$w_j$ are really precomputations of $g$ followed by $f$.)

But, by Lemma~\ref{lem:single_valued}, 
there must be a single $v'$ such that $\rho \Vdash w_i: f(u)=v'$ 
and $\rho \Vdash w_j: f(u)=v'$.
At least one of $v_i$ and $v_j$ is distinct from~$v'$;
without loss of generality suppose $v_i$ is. 
Let $v=v_i$ and $w=w_i$. Thus
we have $\rho \Vdash w: f(g(v)) = v'$ and $v'\neq v$,
as required.
\end{proof}

Now consider the following  Prover-Adversary game, given by
an $\NP$ search problem $Q(x,y)$ and a $\Sigma^p_2$ search problem $R(x',y')$.
At the start of the game, the Prover queries $R$ for some input $x'$, 
with $|x'|$ polynomial  in $|x|$, and the Adversary
gives a reply $y'$. Then the Prover repeatedly queries bits of the oracle~$\alpha$, and the
Adversary replies. The Prover is limited in the number of bits of $\alpha$ he can remember at once, and can also forget bits to save memory. 
The Prover wins when the partial oracle in his memory
either witnesses $Q(x,y)$ for some~$y$, or witnesses that $R(x',y')$ is false.
This game models $\PLS$ counterexample reducibility, in the following sense.

\begin{lemma} \label{lem:P_A_game}
Suppose an $\NP$ search problem $Q(x,y)$ is \changed{cleanly}
$\PLS$ counterexample
reducible to a $\Sigma^p_2$ search problem $R(x',y')$. Then for all inputs $x$ the Prover
can win the  game
using only polynomially many (in $|x|$) bits of memory.
\end{lemma}

\begin{proof}
This is an immediate consequence of the definitions of PLS counterexample
reducibility (Definition~\ref{def:PLS_counterexample}) and  PLS. Using the notation
of Definition~\ref{def:PLS_counterexample},
the Prover first asks for~$y'$ such that $R(e(x), y')$. 
\changed{Since the reduction is clean, the value of $e(x)$ does not depend on any oracle queries
(Definition~\ref{def:clean_reducibility}).}
He then 
sets $x'' = \ang{x,y'}$ and
simulates the (exponential time, but polynomial memory) task of solving the PLS problem~$P(x'',y'')$
by starting with $y''=0$ and then repeatedly setting~$y''$ to~$N_{x''}(y'')$,
finding domain elements of smaller and
smaller cost, until either the costs stop decreasing or $y''$ leaves the domain $F_{x''}$.
\changed{Once~$y''$ satisfying $P(x'', y'')$ is found,
the Prover simulates the polynomial-time computation of~$y := d(x,y'')$.
Then he simulates two additional polynomial-time computations
in order to check whether $Q(x,y)$ holds and whether $y$ witnesses that $R(x',y')$ is false.
By the definition of PLS counterexample reducibility, one of the two possibilities must hold,
thus allowing the Prover to win the game.}

\changed{During the entire game, the Prover} never needs to remember more
bits of the oracle than are necessary to fix simultaneously the 
cost and membership of the domain of one \changed{candidate solution to $P$},
the computation of its neighbour, the cost of the neighbour, 
and possibly \changed{computations of $d$ and} the witnessing for $Q$ and~$R$.
\end{proof}

We can now prove Theorem~\ref{thm:cpls-not-in-ourclass}, that CPLS is not in the class $\APPROX$.

\begin{proof}[Proof of Theorem ~\ref{thm:cpls-not-in-ourclass}]
Assume that CPLS is in $\APPROX$. \changed{Then by Corollary~\ref{cor:clean} it is 
cleanly PLS counterexample reducible to an instance of $\rWPHP_2$  
given by some functions $f$ and~$g$.}
This means that, by Lemma~\ref{lem:P_A_game}, 
the Prover can win the Prover-Adversary game 
in which $Q$ is $\CPLS$ and $R$ is $\rWPHP_2$, 
using only polynomially many bits of memory.
We obtain a contradiction by describing a strategy for 
the Adversary that defeats any Prover with small memory.

The Prover first makes his query $x'$ to $\rWPHP_2$. 
The Adversary then picks a restriction 
$\rho$ from $\Rpq$ for which there exist a precomputation $w$ and 
numbers $v,v' < 2x'$ with $v\neq v'$ such
that $\rho \Vdash w: f_{x'}(g_{x'}(v)) = v'$. By Lemma~\ref{lem:good-rho} 
such a $\rho$ exists, and by Lemma~\ref{lem:prob_bad} we may
further assume that it is good. The Adversary replies with $\ang{v,w}$.

Then, using the \changed{goodness of $\rho$ and the limited size} of the Prover's memory,
the Adversary is able to have in hand throughout the game a legal $\sigma \supseteq \rho$
which contains all bits in the Prover's current memory. 
Such a $\sigma$ can never witness $\CPLS$, because it is legal.
However, a legal $\sigma$ also cannot witness that $\ang{v,w}$ 
is not a solution to $\rWPHP_2$, 
because the only way to do this would be to witness
that one of the NO replies in $w$ is wrong, which is
impossible by the choice of~$\rho$.
The details of the strategy are as in the proof
of \cite[Theorem~5.10]{RRR}.
\end{proof}

\section{Reformulation in propositional logic}
\label{sec:propositional}

In this section we sketch another way of presenting our main result about bounded arithmetic, 
that CPLS, considered as a $\forall \Sigma^b_1$ principle, is not provable in $\APC_2$. 
We will use propositional proof complexity and in particular the well-known
Paris-Wilkie translation of relativized bounded arithmetic  into propositional logic~\cite{paris_wilkie}.

Suppose $\phi$ is bounded formula of $\LPV$, and that we have
specified values~$\bar{n}$ for all free variables in $\phi$. 
We can write a propositional formula $\ang{\phi}$ with the same semantics as $\phi$,
if we interpret propositional variables $x_n$ as bits $\alpha(n)$ of the oracle.
Below we will use \emph{narrow} to mean ``of width polylogarithmic in $\bn$''.

If $\phi$ does not mention the oracle $\alpha$, then its translation $\ang{\phi}$ is the 
propositional constant $\top$ or $\bot$, depending on whether~$\phi$ is true or
false in $\N$. If $\phi$ is~$\alpha(n)$, then $\ang{\phi}$ is the propositional variable $x_n$.
If $\phi$ is a $\PV$ formula, then $\ang{\phi}$ is a narrow CNF 
--- we can take it to be the conjunction of clauses expressing ``some oracle reply in $w$ is false'' 
over all possible rejecting computations $w$ of the polynomial-time machine deciding~$\phi$.
If~$\phi$ is a $\Pi^b_1$ formula $\ab{x}{n} \theta(x)$, 
then again $\ang{\phi}$ is a narrow CNF,
namely the conjunction, over $m<n$, of the translations $\ang{\theta(x)}$ with $x \mapsto m$.

The translation theorem we will use follows from the translation of $T^1_2$ into treelike Res(log) refutations from~\cite{krajicek:WPHP} and~the connection between treelike Res(log) and narrow resolution~\cite{Lauria:treelike_narrow}.
It can also be shown via PLS witnessing, as described in~\cite{BKT2014}.

\begin{theorem} \label{the:PW_translate}
Let $\phi(\bar{n})$ be a $\Pi^b_1$ formula and suppose $T^1_2 \vdash \forall \bar{n}\, \neg \phi(\bar{n})$.
Then the translations $\ang{\phi}$ have narrow resolution refutations.
\end{theorem}
 
Now suppose for a contradiction that $\APC_2 \vdash \CPLS$.
Consider 
\changed{$\CPLS$ as described} in Section~\ref{sec:fixing},
with parameters $a=b=n$ and $c=\lfloor n^{1/7} \rfloor$ 
and the structure of the problem given entirely by the oracle. 
Let $Q(n,y)$ assert that $y$ is a solution to \changed{CPLS on input $n$}.
We may bound $y$ by some term $t(n)$, such that
$\APC_2 \vdash \forall n \eb{y}{t} Q(n,y)$.
By the proof of Lemma~\ref{lem:APC_to_counterexample},
there exist $\P^{\NP}$ machines $f,g$ defining an instance of $\rWPHP_2$, and a term $s(n)$,  such that
\begin{multline*}
T^1_2 \vdash \forall n \ab{v}{2s} \forall w \, 
[w \textrm{ is not a computation of } f_s(g_s(v)) \lor \mathop{\mathrm{output}}(w) = v    \\
~~ \quad \lor \eb{y}{t} Q(n,y)]. \hspace{-10pt}
\end{multline*}

Let $M$ be the $\P^\NP$ machine which takes input $n, v$ and computes 
$f_s(g_s(v))$ by first computing $g$ and then $f$. 
We think of $v$ as the ``real input" to~$M$ and of~$n$ as a parameter,
and write $\Comp_M(v,w)$ for the $\Pi^b_1$ formula 
from Definition~\ref{def:P_NP_computation} expressing that $w$ is a computation of 
$M$ on input $v$.
Noting that the  expression on the right above
is~$\forall \Sigma^b_1$, we can apply Theorem~\ref{the:PW_translate}
to conclude that the family of narrow CNFs 
\[
\Phi_{n,v,w} :=
\ang{ v < 2s }
\wedge
\ang{\Comp_M(v,w)} 
  \wedge
  \ang{\mathrm{output}(w) \neq v}
\wedge
\bigwedge_{y<t} \ang{ \neg Q(n,y)}
\]
has narrow resolution refutations, that is, of width polylogarithmic in $n,v,w$.

Fix a suitable large $n$.
By definition, no legal restriction $\sigma$ can falsify any clause in the last conjunct
$\bigwedge_{y<t} \ang{ \neg Q(n,y)}$, as otherwise
for some $y$ there is an accepting computation of $Q(n,y)$ over $\sigma$,
so $\sigma$ witnesses CPLS.

By Lemma~\ref{lem:good-rho}, 
with high probability for a random $\rho$ from $\Rpq$
there exist
$v<2s$ and a precomputation~$w$ of $M$ on~$s$ with output$(w)\neq v$
such that $w$
is fixed by $\rho$, meaning that all witnesses in $w$ to YES answers 
are correct in $\rho$ and no query with a NO answer has a witness in any legal 
extension of $\rho$. It follows that \changed{for such $v,w$,}
no clause in the first three conjuncts \changed{of $\Phi_{n,v,w}$}
is false in any legal extension of $\rho$.
By Lemma~\ref{lem:prob_bad} we can pick a good $\rho$ for which such $v,w$ exist.

By the Prover-Adversary construction in the proof
of \cite[Theorem~5.10]{RRR}, we can exploit the limited width
of the refutation of $\Phi_{n,v,w}$ to find a legal extension of $\rho$
which falsifies one of the conjuncts of $\Phi_{n,v,w}$. This is 
a contradiction.

\section{Open problems}\label{sec:open-problems}
The \emph{random resolution} propositional proof system
was introduced in~\cite{BKT2014}. Very roughly speaking, 
a refutation of a CNF $F$ in this system is a refutation of $F \wedge A$,
where $A$ is any CNF which is true with high probability.

Suppose a sentence $\forall n \, \phi(n)$, with $\phi$ a $\Sigma^b_1$ formula, 
is provable in the subtheory of $\APC_2$ consisting
of $T^1_2$ together with the surjective WPHP only for 
polynomial time functions. It was
shown in~\cite{BKT2014} that this implies that the 
translations $\ang{\neg \phi(n)}$ have narrow refutations in random 
resolution.

\begin{open}
Is there a natural propositional proof system which captures,
in a similar way, the $\forall \Sigma^b_1$ consequences of 
full $\APC_2$?
\end{open}

Ideally, one would want to show not only that 
$\APC_2$ proofs translate into the system, 
but also something in the opposite direction, 
for example, that if $\ang{\neg \phi(n)}$ has 
small, suitably uniform refutations in the system,
then $\forall n \, \phi(n)$ is provable in $\APC_2$.
Some system with these properties could be constructed
using the Paris-Wilkie translation and our arguments in Section \ref{sec:propositional},
but it would be rather unnatural and awkward.

It is consistent with what we know that narrow random resolution,
or possibly random resolution with no width restriction, already provides a positive answer to Open Problem~1.
So, we can ask:
\begin{open}
Is there a $\forall \Sigma^b_1$ sentence which is provable in $\APC_2$
but whose propositional translations do not have narrow random resolution refutations?
\end{open}
A candidate is the Herbrandized ordering principle HOP,
which is provable in $\APC_2$~\cite{BKT2014} but not in the
subtheory mentioned above~\cite{atserias_thapen}.

What makes this problem interesting is that, so far, our only tool for
proving lower bounds on random resolution is the fixing
lemma of~\cite{RRR}. For a typical random restriction, 
it is a  small step from proving this
to proving our conditional fixing lemma
from Section~\ref{sec:fixing}, which implies unprovability in $\APC_2$.
But showing a separation seems to require finding a
principle and a random restriction for which one lemma holds, but not the other.
The restrictions used to show unprovability of HOP in~\cite{atserias_thapen}
may be useful here.

Finally we mention a rather obvious question: is every problem in $\APPROX$ reducible to CPLS?
This is subsumed in the old open problem, discussed in the introduction,
of separating the classes $\GI_k$ or the theories $T^k_2$:
it is possible that every search problem reducible to any $\GI_k$
is already reducible to CPLS.

\end{document}